 \documentclass[12pt, a4paper]{amsart}
\usepackage{amscd, amssymb,amsmath,upref, color}

\usepackage[all]{xy}
\usepackage{amscd}
\allowdisplaybreaks[2]
\begin{document}
\newcommand{\Chi}{\raisebox{2pt}{\ensuremath{\chi}}}
\def\C{\mathbb{C}}
\def\N{\mathbb{N}}
\def\Z{\mathbb{Z}}
\def\R{\mathbb{R}}
\def\T{\mathbb{T}}
\def\supp{\operatorname{supp}}
\def\Aut{\operatorname{Aut}}
\def\End{\operatorname{End}}
\def\Ad{\operatorname{{Ad}}}
\def\Tr{\operatorname{Tr}}
\def\id{\operatorname{id}}
\def\lsp{\operatorname{span}}
\def\clsp{\operatorname{\overline{span}}}
\def\To{\mathcal{T}}

\def\H{\mathcal{H}}
\def\B{\mathcal{B}}
\def\F{\mathcal{F}}

\def\K{\mathcal{K}}
\def\TT{\mathcal{T}}
\def\QQ{\mathcal{Q}}
\def\AA{\mathcal{A}}
\def\EE{\mathcal{E}}
\def\L{\mathcal{L}}
\def\O{\mathcal{O}}
\def\OO{\mathcal{O}}
\renewcommand{\Chi}{\raisebox{2pt}{\ensuremath{\chi}}}
\newcommand{\under}{\backslash}
\newcommand{\qregular}{regular }

\renewcommand{\theenumi}{\alph{enumi}} 
\renewcommand{\theenumii}{\roman{enumii}}

\newtheorem{thm}{Theorem}  [section]
\newtheorem{cor}[thm]{Corollary}
\newtheorem{lemma}[thm]{Lemma}
\newtheorem{prop}[thm]{Proposition}
\theoremstyle{definition}
\newtheorem{defn}[thm]{Definition}
\newtheorem{remark}[thm]{Remark}
\newtheorem{example}[thm]{Example}
\newtheorem{keyexample}[thm]{Key Example}
\newtheorem{remarks}[thm]{Remarks}

\numberwithin{equation}{section}

\title[Exel and Stacey crossed products, and Cuntz-Pimsner algebras]
{\boldmath{Exel crossed products, Stacey crossed products,\\ and Cuntz-Pimsner algebras}}

\author[an Huef]{Astrid an Huef}

\author[Raeburn]{Iain Raeburn}
\address{Department of Mathematics and Statistics\\ University of Otago\\ PO Box 56\\Dunedin 9054\\ New Zealand}
\email{astrid@maths.otago.ac.nz, iraeburn@maths.otago.ac.nz}

\begin{abstract}
There are many different crossed products by an endomorphism of a $C^*$-algebra, and constructions by Exel and Stacey have proved particularly useful. Here we show that every Exel crossed product is isomorphic to a Stacey crossed product (though by a different endomorphism of a different $C^*$-algebra), that every Stacey crossed product is an Exel crossed product, and answer some questions raised by Ionescu and Muhly.

While this manuscript is not yet in its final form(s) several people have already used our results, so we are posting it now.
\end{abstract}

\date{6 July 2011}
\maketitle

\section{Introduction}
Everybody agrees that when $\alpha$ is an automorphism of a unital $C^*$-algebra $A$, the crossed product $A\rtimes_\alpha\Z$ is generated by a unitary element $u$ and a representation $\pi:A\to A\rtimes_{\alpha}\Z$ which satisfy the covariance relation
\begin{equation}\label{covauto}
\pi(\alpha(a))=u\pi(a)u^*\ \text{ for $a\in A$.}
\end{equation}
The covariance relation can be reformulated as $\pi(\alpha(a))u=u\pi(a)$ or $u^*\pi(\alpha(a))u=\pi(a)$, and, when $\alpha$ is an automorphism, these reformulations are equivalent to \eqref{covauto}. When $\alpha$ is an endomorphism, though, these reformulations are no longer equivalent, and give different crossed products. Thus there are several crossed products based on covariance relations in which $u$ is an isometry \cite{Mur, S, E}, and still more crossed products in which $u$ is a partial isometry \cite{Lin,ABL}. The crossed products constructed by Stacey \cite{S} and by Exel \cite{E} have proved to be particularly useful. 

Stacey's crossed product $A\times_\alpha \N$ is generated by an isometry $s$ and a representation $\pi$ of $A$ satisfying $\pi(\alpha(a))=s\pi(a)s^*$. His motivating example was the endomorphism $\alpha$ of the UHF core $A$ in the Cuntz algebra $\OO_n$ described by Cuntz in \cite{C}, for which we recover $\OO_n$ as $A\times_\alpha \N$ (see also  \cite{P, BKR}). Stacey's construction has been extended to semigroups of endomorphisms, and these semigroup crossed products were used to study Toeplitz algebras \cite{ALNR,LR1}; they have since been used extensively in the analysis of $C^*$-algebras arising in number theory (see \cite{LR2, Bren, L, Li, LNT}, for example).

Exel's construction depends on the choice of a transfer operator $L:A\to A$ for $\alpha$, which is a positive linear map satisfying $L(\alpha(a)b)=aL(b)$. He uses $L$ to build a Hilbert bimodule $M_L$ over $A$, and then his crossed product $A\rtimes_{\alpha,L}\N$ is closely related to the Cuntz-Pimsner algebra $\OO(M_L)$ of this bimodule (the precise relationship is described in \cite{BR}). The motivating example for Exel's construction is the endomorphism of $C(\{1,\cdots,n\}^\infty)$ induced by the backward shift, for which averaging over the $n$ preimages of each point gives a transfer operator $L$ such that $\OO_n\cong C(\{1,\cdots,n\}^\infty)\rtimes_{\alpha,L}\N$. More generally, Exel realised each Cuntz-Krieger algebra $\OO_A$ as a crossed product by the corresponding subshift of finite type, thereby giving a very direct proof that the Cuntz-Krieger algebra is determined up to isomorphism by the subshift. Exel's construction has attracted a good deal of attention in connection with irreversible dynamics \cite{EV, EaHR, CS}, and has also been extended to semigroups of endomorphisms with interesting consequences \cite{Lar, BaHLR}.

The purpose of this paper is to discuss some new relationships between the constructions of Stacey and Exel. On the face of it, their constructions are quite different, and are interesting for different classes of endomorphisms: for example, Stacey crossed products are not interesting for unital endomorphisms whereas Exel crossed products are. Nevertheless, we have noticed that $C^*$-algebras can have several different descriptions as crossed products by endomorphisms (as discussed for $\OO_n$ above). Our interest in this subject arose from recent work of Ionescu and Muhly \cite{IM} which describes two different realisations of a particular groupoid $C^*$-algebra $C^*(G)$ as Cuntz-Pimsner algebras. We recognised that both these Cuntz-Pimsner algebras are Exel crossed products, and asked ourselves whether there is a general mechanism at work. 

We think that we have found some interesting relationships between Exel crossed products and Stacey crossed products which explain the phenomenon we observed in \cite{IM}. Our results say that, modulo some minor extra hyotheses on $(A, \alpha, L)$:
\begin{itemize}
\item  Every relative Cuntz-Pimsner algebra of an Exel system is a Stacey crossed product (Theorem~\ref{relCPalg=St}, generalising \cite[Theorem~6.5]{E2}).  In particular, every Exel crossed product is a Stacey crossed product (Corollary~\ref{Ex=St2} and Theorem~\ref{CP=St}).
\item Every Stacey crossed product is an Exel crossed product (Theorem~\ref{extendExel}, generalising \cite[Theorem~4.7]{E}).
\end{itemize}
Combining them gives our explanation of the Ionescu-Muhly phenomenon concerning  the groupoid $C^*$-algebra $C^*(G)$ (see \S\ref{sec-IM}). Ionescu and Muhly also ask whether, with some additional hypotheses,  $C^*(G)$ is  isomorphic to a certain Stacey crossed product of multiplicity $n$;  in an appendix we give an example  where this is not the case.  

Exel crossed products have recently been used to model and study the $C^*$-algebras of directed graphs \cite{BRV}. Since allowing infinite graphs gave important extra generality for graph algebras, this has highlighted the need to study Exel crossed products for nonunital algebras, and we do this throughout. At the end, we apply our results to graph algebras, and find a new realisation of the graph algebra $C^*(E)$ as a Stacey crossed product $C^*(E)^\gamma\times_\beta\N$ by an endomorphism $\beta$ of the core, extending work of Kwa\'sniewski on finite graphs \cite{Kwa} (see \S\ref{ckex}). Our analysis requires a concrete description of the core in the $C^*$-algebra of a column-finite graph, which we provide in an appendix.


\section{Stacey crossed products}

Suppose that $\alpha $ is an endomorphism of a $C^*$-algebra $A$. A \emph{Stacey-covariant representation} of $(A,\alpha)$ in a $C^*$-algebra $B$ consists of a nondegenerate homomorphism $\pi:A\to B$ and an isometry $V\in M(B)$ such that $\pi(\alpha(a))=V\pi(a)V^*$. Stacey showed in \cite[\S3]{S} that there is a crossed product $A\times_\alpha\N$ which is generated by a universal Stacey-covariant representation $(i_A,v)$. If $(\pi,V)$ is a Stacey-covariant representation of $(A,\alpha)$ in $B$, then we write $\pi\times V$ for the nondegenerate homomorphism of $A\times_\alpha \N$ into $B$ such that $(\pi\times V)\circ i_A=\pi$ and $(\pi\times V)(v)=V$. (Stacey called $A\times_\alpha\N$  ``the multiplicity-one crossed product" of $(A,\alpha)$.) 

The crossed product $A\times_\alpha\N$ carries a dual action $\hat\alpha$ of $\T$, which is characterised by $\hat\alpha_z(i_A(a))=i_A(a)$ and $\hat\alpha_z(v)=zv$. The following ``dual-invariant uniqueness theorem'' says that this dual action identifies $A\times_\alpha\N$ among $C^*$-algebras generated by Stacey-covariant representations of $(A,\alpha)$. It was basically proved in \cite[Proposition~2.1]{BKR} (modulo the correction made in \cite{ALNR}).  

\begin{prop}\label{diut}
Suppose that $\alpha$ is an endomorphism of a $C^*$-algebra $A$, and $(\pi,V)$ is a Stacey-covariant representation of $(A,\alpha)$ in a $C^*$-algebra $D$. If $\pi$ is faithful and there is a strongly continuous action $\gamma:\T\to \Aut D$ such that $\gamma_z(\pi(a))=\pi(a)$ and $\gamma_z(V)=zV$, then $\pi\times V$ is faithful on $A\times_\alpha\N$.
\end{prop}

\begin{proof}
The conditions on $\gamma$ say that $(\pi\times V)(\hat\alpha_z(b))=\gamma_z((\pi\times V)(b))$ for all $b\in A\times_\alpha\N$. Thus
\begin{align*}
\Big\|(\pi\times V)\Big(\int_{\T}\hat\alpha_z(b)\,dz\Big)\Big\|&=\Big\|\int_{\T}(\pi\times V)(\hat\alpha_z(b))\,dz\Big\|=\Big\|\int_{\T}\gamma_z((\pi\times V)(b))\,dz\Big\|\\
&\leq \int_{\T}\|\gamma_z((\pi\times V)(b))\|\,dz=\int_{\T}\|(\pi\times V)(b)\|\,dz\\
&=\|(\pi\times V)(b)\|,
\end{align*}
We now take $(B,\beta)=(A\times_\alpha\N,\hat\alpha)$, and apply \cite[Lemma~2.2]{BKR} to $(B,\beta)$. We have just verified the hypothesis (2) of \cite[Lemma~2.2]{BKR}. The other hypothesis (1) asks for $\pi\times V$ to be faithful on the fixed-point algebra $B^\beta=(A\times_\alpha\N)^{\hat\alpha}$. However, the proof of \cite[Lemma~1.5]{ALNR} uses neither that $A$ is unital nor the estimate (ii) in \cite[Theorem~1.2]{ALNR}, and hence we can deduce from that proof that $\pi\times V$ is faithful on $B^\beta$. Thus \cite[Lemma~2.2]{BKR} applies, and the result follows.
\end{proof}


\section{Relative Cuntz-Pimsner algebras of Exel systems}\label{mainthm}

An endomorphism $\alpha$ of a $C^*$-algebra $A$ is \emph{extendible} if it extends to a strictly continuous endomorphism $\overline{\alpha}$ of $M(A)$. Nondegenerate endomorphisms, for example, are autmatically extendible with $\overline{\alpha}(1)=1$. In this paper we are interested in \emph{Exel systems} $(A,\alpha,L)$ of the kind studied in \cite{BRV}, which means that $\alpha$ is an extendible endomorphism of a $C^*$-algebra $A$  and  $L:A\to A$ is a  positive linear map which  extends to a  positive linear map $\overline{L}:M(A)\to M(A)$ such that 
\begin{equation}\label{transferoponmofa}
L(\alpha(a)m)=a\overline{L}(m)\ \text{ for $a\in A$ and $m\in M(A)$.}
\end{equation}
Equation~\eqref{transferoponmofa} implies that $\overline{L}$ is strictly continuous, and we then assume further that $\overline{L}(1_{M(A)})=1_{M(A)}$ (but not that $\overline{\alpha}$ is unital).  We say that $L$ is a \emph{transfer operator} for $(A, \alpha)$.

Let $M_L$ denote the Hilbert bimodule over $A$ constructed in \cite{BRV}. (This construction extends the one  of \cite{E} to non-unital $A$, and follows the lines of \cite{BR}.)  Briefly,  $A$ is given a bimodule structure by $a\cdot m=am$ and $m\cdot b=m\alpha(b)$, and the pairing $\langle m\,,\, n\rangle= L(m^*n)$ defines a pre-inner product on $A$.  Modding out by $m$ such that $\langle m\,,\, m\rangle=0$ and completing gives a right-Hilbert bimodule $M_L$. We denote by $q:A\to M_L$ the canonical map of $A$ onto a dense sub-bimodule of $M_L$, and by $\phi$ the homomorphism of $A$ into $\L(M_L)$ implementing the left action.  

Following \cite{FR}, a representation $(\psi,\pi)$ of $M_L$ in a $C^*$-algebra $B$  consists of a linear map $\psi:M_L\to B$ and a homomorphism $\pi:A\to B$ such that
\[
\psi(m\cdot a)=\psi(m)\pi(a),\quad \psi(m)^*\psi(n)=\pi(\langle m\,,\, n\rangle), \text{\ and\ }\psi(a\cdot m)=\pi(a)\psi(m)
\]
for $a\in A$ and $m,n\in M_L$. By \cite[Proposition~1.8]{FR}, a representation $(\psi,\pi)$ of $M_L$ gives a representation $(\psi^{\otimes i},\pi)$ of $M_L^{\otimes i}:=M_L\otimes_A\cdots\otimes_A M_L$ such that $\psi^{\otimes i}(m_1\otimes_A\cdots\otimes_A m_i)=\psi(m_1)\cdots\psi(m_i)$. The Toeplitz algebra $\TT(M_L)$  is the $C^*$-algebra genertaed by a universal  representation $j_M,j_A)$ of $M_L$, and \cite[Lemma~2.4]{FR} says that
\[
\TT(M_L)=\clsp\big\{j_M^{\otimes i}(m)j_M^{\otimes j}(n)^*:m\in M_L^{\otimes i},\;n\in M_L^{\otimes j}\;i,j\in\N\big\}.
\]
There is  a strongly continuous action $\gamma:\T\to\Aut \TT(M_L)$, called the \emph{gauge action}, such that $\gamma_z(j_A(a))=j_A(a)$ and $\gamma_z(j_M(m))=zj_M(m)$ \cite[Proposition~1.3]{FR}.

A representation $(\psi,\pi)$ of $M_L$ in $B$ gives  a homomorphism $(\psi,\pi)^{(1)}:\K(M_L)\to B$ such that $(\psi,\pi)^{(1)}(\Theta_{m,n})=\psi(m)\psi(n)^*$ for every rank-one operator $\Theta_{m,n}:p\mapsto m\cdot\langle n\,,\, p\rangle$. Following \cite[\S1]{FMR}, if $J$ is an ideal of $A$ contained in $\phi^{-1}(\K(M_L))$, then we view the \emph{relative Cuntz-Pimsner algebra} $\O(J,M_L)$ of \cite{MS} as the quotient of $\TT(M_L)$ by the ideal generated by
\[
\big\{j_A(a)-(j_M, j_A)^{(1)}(\phi(a)): a\in J\big\}.
\]
We write $Q$ or $Q_J$ for the quotient map.  Then $(k_M,k_A):=(Q_J\circ j_M,Q_J\circ j_A)$ is universal for  representations $(\psi,\pi)$ which are \emph{coisometric} on $J$ (that is, satisfy $\pi|_J= (\psi,\pi)^{(1)}\circ\phi|_J$).  If $J=\{0\}$ then $\O(J,M_L)$ is just $\TT(M_L)$; if $J=\phi^{-1}(\K(M_L))$ then $\O(J,M_L)$ is the \emph{Cuntz-Pimsner algebra} $\O(M_L)$ of \cite{Pim}, and representations that are coisometric on $\phi^{-1}(\K(M_L))$ are called \emph{Cuntz-Pimsner covariant}.

It follows from \cite[Lemma~2.4]{FR} that each quotient $\O(J,M_L)$ carries a gauge action $\gamma:\T\to \O(J,M_L)$ such that the quotient map  $Q_J$ is equivariant, and  the fixed-point algebra or \emph{core}  is
\[
\O(J,M_L)^\gamma=\clsp\big\{k_M^{\otimes i}(m)k_M^{\otimes i}(n)^*:m,n\in M_L^{\otimes i},\;i\in\N\big\}.
\]
The following theorem generalises \cite[Theorem~6.5]{E2}.

\begin{thm}\label{relCPalg=St}
Suppose that $(A,\alpha,L)$ is an Exel system, and $J$ is an ideal of $A$ contained in $\phi^{-1}(\K(M_L))$.  
 There is a unique isometry $V\in M(\O(J,M_L))$ such that $k_M(q(a))=k_A(a)V$ for $a\in A$, and $\Ad V$ restricts to an endomorphism $\alpha'$ of the core $C_J:=\O(J,M_L)^\gamma$ such that 
\begin{equation}\label{defalpha'}
\alpha'\big(k_M^{\otimes i}(a\cdot m)k_M^{\otimes i}(b\cdot n)^*\big)
=k_M^{\otimes (i+1)}(q(\alpha(a))\otimes_Am)k_M^{\otimes (i+1)}(q(\alpha(b))\otimes_An)^* 
\end{equation}
for $a,b\in A$ and $m,n\in M_L^{\otimes i}$.
Further,  $\alpha'$ is extendible with $\overline{\alpha'}(1)=VV^*$, is injective and has range $\overline{\alpha'}(1)C_J\overline{\alpha'}(1)$.
Finally,  $(\id,V)$ is a Stacey-covariant representation of $(C_J,\alpha')$ such that $\id\times V$ is an isomorphism of the Stacey crossed product $C_J\times_{\alpha'}\N$ onto $\O(J,M_L)$.
\end{thm}

\begin{proof}
Since we know from \cite[Corollary~3.5]{BRV} that $k_A:A\to \O(J,M_L)$ is nondegenerate\footnote{We caution that this nondegeneracy is not at all obvious, and even slightly surprising, because the representation $\pi$ in a Toeplitz representation $(\psi,\pi)$ is not required to be nondegenerate.}, there is at most one multiplier $V$ satisfying $k_M(q(a))=k_A(a)V$, and we have uniqueness.

When the $C^*$-algebra $A$ has an identity, we can deduce from the results in \cite[\S3]{BR} that $V:=k_M(q(1))$ has the required properties. When $A$ does not have an identity, we take an approximate identity $\{e_{\lambda}\}$ for $A$, and claim, following Fowler \cite[\S3]{F}, that $\{k_M(q(e_{\lambda}))\}$ converges strictly in $M(\O(J,M_L))$ to a multiplier $V$. Indeed, for $a,b\in A$, $m\in M_L^{\otimes i}$ and $n\in M_L^{\otimes j}$ we have
\begin{align}\label{psialcvgeleft}
k_M(q(e_\lambda))k_M^{\otimes i}(a\cdot m)k_M^{\otimes j}(b\cdot n)^* 
&=k_M(q(e_\lambda))k_A(a)k_M^{\otimes i}(m)k_M^{\otimes j}(b\cdot n)^*\\
&=k_M(q(e_\lambda\alpha(a)))k_M^{\otimes i}(m)k_M^{\otimes j}(b\cdot n)^*\notag\\
&\to k_M(q(\alpha(a)))k_M^{\otimes i}(m)k_M^{\otimes j}(b\cdot n)^*,\notag
\end{align}
and similarly
\begin{equation}\label{psialcvgeright}
k_M^{\otimes i}(a\cdot m)k_M^{\otimes j}(b\cdot n)^*k_M(q(e_\lambda))\to k_M^{\otimes i}(a\cdot m)k_M^{\otimes j}(n)^*k_M(q(b^*)).
\end{equation}
Since $\overline{L}$ is positive and $\overline{L}(1)=1$, we have $\|L\|\leq 1$, and $\|q(e_{\lambda})\|\leq \|e_{\lambda}\|\leq 1$ for all $\lambda$. Thus, since the elements $k_M^{\otimes i}(a\cdot m)k_M^{\otimes j}(b\cdot n)^*$ span a dense subspace of $\O(J,M_L)$, an $\epsilon/3$ argument using \eqref{psialcvgeleft} and~\eqref{psialcvgeright} shows that $\{k_M(q(e_{\lambda}))b\}$ and $\{bk_M(q(e_{\lambda}))\}$ converge in $\O(J,M_L)$ for every $b\in \O(J,M_L)$. Thus $\{k_M(q(e_{\lambda}))\}$ is strictly Cauchy, and since $M(\O(J,M_L)$ is strictly complete we deduce that $\{k_M(q(e_{\lambda}))\}$ converges strictly to a multiplier $V$; \eqref{psialcvgeleft} implies that $V$ satisfies
\begin{equation}\label{propV}
Vk_M^{\otimes i}(a\cdot m)k_M^{\otimes j}(b\cdot n)^*=k_M(q(\alpha(a)))k_M^{\otimes i}(m)k_M^{\otimes j}(b\cdot n)^*.
\end{equation}

To see that $V$ is an isometry, we observe that 
\[
k_M(q(e_{\lambda}))^*k_M(q(e_{\lambda}))=k_A(\langle q(e_{\lambda}),q(e_{\lambda})\rangle)=k_A(L(e_{\lambda}^2));
\]
since $\overline{L}$ is strictly continuous, $L(e_{\lambda}^2)$ converges strictly to $\overline{L}(1_{M(A)})=1_{M(A)}$.  Since $k_A:A\to \O(J,M_L)$ is nondegenerate, $k_M(q(e_{\lambda}))^*k_M(q(e_{\lambda}))$ converges strictly to $1=1_{M(\O(J,M_L))}$, and since the multiplication in a multiplier algebra is jointly strictly continuous on bounded sets (by another $\epsilon/3$ argument), we deduce that $V^*V=1$. Thus $V$ is an isometry. Next, we let $a\in A$ and compute
\begin{equation}
k_M(q(a))=\lim_{\lambda}k_M(q(ae_\lambda))=\lim_{\lambda}k_M(a\cdot q(e_\lambda))=\lim_{\lambda}k_A(a)k_M(q(e_\lambda))=k_A(a)V,
\end{equation}
so $V$ has the required properties. 

Conjugating by the isometry $V\in M(\O(J,M_L))$ gives an endomorphism $\Ad V:T\mapsto VTV^*$, and two applications of \eqref{propV} show that
\begin{align}\label{formforalpha'}
\Ad V\big(k_M^{\otimes i}(a\cdot m)k_M^{\otimes i}(b\cdot n)^*\big)
&=k_M(q(\alpha(a)))k_M^{\otimes i}(m)k_M^{\otimes i}(n)^*k_M(q(\alpha(b)))^*\\
&=k_M^{\otimes(i+1)}(q(\alpha(a))\otimes_A m)k_M^{\otimes (i+1)}(q(\alpha(b))\otimes_An)^*.\notag
\end{align}
The formula \eqref{formforalpha'} implies that $\Ad V$ maps the core $C_J$ into itself, and that the restriction $\alpha':=(\Ad V)|_{C_J}$ satisfies \eqref{defalpha'}. The pair $(\id,V)$ is then by definition Stacey covariant for $\alpha'\in \End C_J$, and we can apply the dual-invariant uniqueness theorem (Proposition~\ref{diut}) to the gauge action $\gamma$ on $\O(J,M_L)$, finding that $\id\times V$ is a faithful representation of $C_J\times_{\alpha'}\N$ in $\O(J,M_L)$. The identity $k_M(q(a))=k_A(a)V$ implies that $k_M(M_L)=\overline{k_A(A)V}$, and since $\O(J,M_L)$ is generated by $k_A(A)\subset C_J$ and $k_M(M_L)$, the range of $\id\times V$ contains the generating set $k_A(A)\cup k_A(A)V$, and hence is all of $\O(J,M_L)$.

It remains for us to prove the assertions about $\alpha'$.  It is injective because $\Ad V$ is. Since  $k_A$ is nondegenerate the image $\{k_A(e_\lambda)\}$ of an approximate identity $\{e_\lambda\}$ converges strictly to $1$ in $M(\O(J,M_L))$. Hence $\{\alpha'(k_A(e_\lambda))\}=\{Vk_A(e_\lambda) V^*\}$ converges strictly to the projection $VV^*$. Thus $\alpha'$ is extendible with $\overline{\alpha'}(1)=VV^*$ by, for example, \cite[Proposition~3.1.1]{adji}. 

The range of $\alpha'$ is certainly contained in the corner $\overline{\alpha'}(1)C_J \overline{\alpha'}(1)$. To see the reverse inclusion, fix $T\in \overline{\alpha'}(1)C_J\overline{\alpha'}(1)$. Then $T=VV^*SVV^*=\Ad V(V^*SV)$ for some $S\in C_J$. 
Since $C_J$ is $\alpha'$-invariant, to see that $T$ is in the range of $\alpha'=\Ad V|_{C_J}$  it suffices to see that $V^*SV$ is in $C_J$, and, by continuity of $\Ad V^*$, it suffices to see this for $S$ of the form
\[
S=k_M^{\otimes i}((a\cdot m_1)\otimes_Am')k_M^{\otimes i}((b\cdot n_1)\otimes_A n')^*
\]
where $a, b\in A$, $m_1, n_1\in M_L$ and $m',n'\in M_L^{\otimes (i-1)}$. For $i\geq 1$ the calculation
\begin{align*}
V^*k_M^{\otimes i}((a\cdot m_1)\otimes_Am')
&=(k_A(a^*)V)^*k_M(m_1)k_M^{\otimes (i-1)}(m')\\
&=k_M(q(a^*))^*k_M(m_1)k_M^{\otimes (i-1)}(m')\\
&=k_A(\langle q(a^*),m_1\rangle)k_M^{\otimes (i-1)}(m')\\
&=k_M^{\otimes (i-1)}(\langle q(a^*),m_1\rangle\cdot m')
\end{align*}
gives 
\[
V^*SV=k_M^{\otimes (i-1)}(\langle q(a^*),m_1\rangle\cdot m')k_M^{\otimes (i-1)}(\langle q(b^*),n_1\rangle\cdot n')^*\in C_J.
\]
For $i=0$  we have 
\[V^*SV=V^*k_A(a)k_A(b)^*V=k_M(q(a^*))^*k_M(q(b^*))=k_A(\langle q(a^*)\,,\, q(b^*)\rangle)\in C_J.
\]
Thus $T=\Ad V(V^*SV)=\alpha'(V^*SV)$, and $\alpha'$ has range $\overline{\alpha'}(1)C_J \overline{\alpha'}(1)$. 
\end{proof}


\section{Exel crossed products}

Suppose that $(A,\alpha,L)$ is an Exel system, as in \S\ref{mainthm}.
As in \cite{BRV}, a \emph{Toeplitz-covariant representation} of $(A,\alpha,L)$ in a $C^*$-algebra $B$ consists of a nondegenerate homomorphism $\pi:A\to B$ and an element $S\in M(B)$ such that
\[S\pi(a)=\pi(\alpha(a))S\quad\text{and}\quad
S^*\pi(a)S=\pi(L(a)).
\]
The \emph{Toeplitz crossed product} $\TT(A,\alpha,L)$ is generated by a universal Toeplitz-covariant representation $(i,s)$ (or, more correctly, by $i(A)\cup i(A)s$).  By \cite[Proposition~3.1]{BRV}, there is a map $\psi_s:M_L\to \TT(A,\alpha,L)$ such that $\psi_s(q(a))=i(a)s$, and $(\psi_s,i)$ is a  representation of $M_L$. Set 
\[
K_\alpha:=\overline{A\alpha(A)A}\cap \phi^{-1}(\K(M_L)).
\]
 In \cite[\S4]{BRV}, the \emph{Exel crossed product} $A\rtimes_{\alpha,L}\N$ of a possibly non-unital $C^*$-algebra $A$ by $\N$ is the quotient of $\TT(A,\alpha,L)$ by the ideal generated by
\[
\big\{i(a)-(\psi_s,i)^{(1)}(\phi(a)): a\in K_\alpha\big\}.
\]
We write $\QQ$ for the quotient map of $\TT(A,\alpha,L)$ onto $A\rtimes_{\alpha,L}\N$, and $(j,t):=(\QQ\circ i,\overline{\QQ}(s))$. There is a \emph{dual action} $\hat\alpha$ of $\T$ on $A\rtimes_{\alpha,L}\N$ such that $\hat\alpha_z(j(a))=j(a)$ and $\hat\alpha_z(t)=zt$.

Theorem~4.1 of \cite{BRV} says that there is an isomorphism $\theta$ of $\O(K_\alpha,M_L)$ onto $A\rtimes_{\alpha,L}\N$ such that $\theta\circ k_A=\QQ\circ i$ and $\theta\circ k_M=\QQ\circ \psi_s$. We will now use $\theta$ to transfer the conclusions of Theorem~\ref{relCPalg=St} over to $A\rtimes_{\alpha,L}\N$.

If $\{e_\lambda\}$ is an approximate idenity for $A$, then
\[
\theta\circ k_M(q(e_\lambda))=\QQ(\psi_s(q(e_\lambda)))=\QQ(i(e_\lambda)s)=\QQ(i(e_{\lambda}))\overline{\QQ}(s)
\]
converges by nondegeneracy of $i$ to $\overline{\QQ}(s)$, and hence $\overline{\theta}$ carries the isometry $V$ of Theorem~\ref{relCPalg=St} into $t:=\overline{\QQ}(s)$.  
The isomorphism $\theta$ is equivariant for the gauge action $\gamma$ on $\O(K_\alpha,M_L)$ and the dual action $\hat\alpha$ on $A\rtimes_{\alpha, L}\N$, and hence maps the core $C_{K_\alpha}=\O(K_\alpha,M_L)^\gamma$ onto $(A\rtimes_{\alpha,L}\N)^{\hat\alpha}$. 

Next, we want a workable description of $(A\rtimes_{\alpha,L}\N)^{\hat\alpha}$. 
For $m=q(a_1)\otimes_A\cdots\otimes_Aq(a_i)\in M_L^{\otimes i}$ we have
\begin{align}\label{simplifyinEcp}
k_M^{\otimes i}(m)&=k_M(q(a_1))\cdots\ k_M(q(a_i))=k_A(a_1)V\cdots k_A(a_i)V\\
&=k_A(a_1)k_A(\alpha(a_2))V^2k_A(a_3)V\cdots k_A(a_i)V\notag\\
&=k_A\big(a_1\alpha(a_2)\alpha^{2}(a_3)\cdots \alpha^{i-1}(a_i)\big)V^i,\notag
\end{align}
and hence $\theta$ takes $k_M^{\otimes i}(m)$ into an element of the form $j(a)t^i$. Now
\[
\lsp\big\{j(a)t^it^{*k}j(b):a,b\in A,\;i,k\in\N\big\}
\]
is a $*$-subalgebra of $A\rtimes_{\alpha,L}\N$ containing the generating set $j(A)\cup j(A)t$, and hence is dense in $A\rtimes_{\alpha,L}\N$. The expectation onto $(A\rtimes_{\alpha,L}\N)^{\hat\alpha}$ is continuous and kills terms with $i\not= k$, so 
\begin{equation}\label{presentfpas}
(A\rtimes_{\alpha,L}\N)^{\hat\alpha}=\clsp\big\{j(a)t^it^{*i}j(b):a,b\in A,\;i\in\N\big\}.
\end{equation}

\begin{cor}\label{Ex=St2}
Suppose that $(A,\alpha,L)$ is an Exel system. Then there is an injective endomorphism $\beta$ of $(A\rtimes_{\alpha,L}\N)^{\hat\alpha}$ such that
\begin{equation}\label{formforbeta}
\beta(j(a)t^it^{*i}j(b))=j(\alpha(a))t^{i+1}t^{*(i+1)}j(\alpha(b)).
\end{equation}
The endomorphism $\beta$ is extendible with $\overline\beta(1)=tt^*$ and has range $tt^*(A\rtimes_{\alpha,L}\N)^{\hat\alpha}tt^*$. The pair $(\id,t)$ is a Stacey-covariant representation of $\big((A\rtimes_{\alpha,L}\N)^{\hat\alpha},\beta\big)$, and $\id\times t$ is an isomorphism of the Stacey crossed product $(A\rtimes_{\alpha,L}\N)^{\hat\alpha}\times_\beta\N$ onto the Exel crossed product $A\rtimes_{\alpha,L}\N$.
\end{cor}

\begin{proof}
Applying the isomorphism $\theta:\O(K_\alpha, M_L)\to A\rtimes_{\alpha, L}\N$ to the conclusion of Theorem~\ref{relCPalg=St} gives an endomorphism $\beta:=\theta\circ\alpha'\circ\theta^{-1}$ of $(A\rtimes_{\alpha,L}\N)^{\hat\alpha}$ and an isomorphism $\id\times t$ of $(A\rtimes_{\alpha,L}\N)^{\hat\alpha}\times_\beta\N$ onto $A\rtimes_{\alpha,L}\N$. It remains for us to check the formula for $\beta$. Let $m=q(a_1)\otimes_A\cdots\otimes_Aq(a_i)$. Then the calculation \eqref{simplifyinEcp} shows that $\theta$ carries $k_M^{\otimes i}(c\cdot m)$ into $j(ca_1\alpha(a_2)\cdots\alpha^{i-1}(a_i))t^i$, and $k_M^{\otimes(i+1)}(q(\alpha(c))\otimes_Am)$ into 
\[
j\big(\alpha(c)\alpha(a_1)\alpha^2(a_2)\cdots\alpha^i(a_i)\big)t^{i+1}=j\big(\alpha(ca_1\alpha(a_2)\cdots\alpha^{i-1}(a_i))\big)t^{i+1},
\]
so \eqref{formforbeta} follows from \eqref{defalpha'}.
\end{proof}

\begin{remark}\label{idcpCP}
Since we have identified how the action $\alpha'$ on the core of  $\O(K_\alpha, M_L)$ pulls over to the Exel crossed product $A\rtimes_{\alpha, L}\N$, we will from now on freely identify $\O(K_\alpha, M_L)$ and $A\rtimes_{\alpha, L}\N$, and drop the isomorphism $\theta$ from our notation.
\end{remark}

\begin{example}
We now discuss a family of Exel systems studied in \cite{EaHR} and \cite{LRR}. Let $d\in \N$ and fix $B\in M_d(\Z)$ with nonzero determinant $N$. (This matrix $B$ plays the same role as the matrix $B$ in \cite{EaHR,LRR}; because $A$ is already heavily subscribed in this paper, we write $B^t$ in place of the matrix $A$ used there.)

The map $\sigma_{B^t}:\T^d\to \T^d$ characterised by $\sigma_{B^t}(e^{2\pi ix})=e^{2\pi iB^tx}$ is a covering map, and induces an endomorphism $\alpha_{B^t}:f\mapsto f\circ\sigma_{B^t}$ of $C(\T^d)$, and $L(f)(z):=N^{-1}\sum_{\sigma_{B^t}(w)=z}f(w)$ defines a transfer operator $L$ for $\alpha_{B^t}$. Proposition~3.3 of \cite{LRR} says that the Exel crossed product $C(\T^d)\rtimes_{\alpha_{B^t},L}\N$ is the universal $C^*$-algebra generated by a unitary representation $u$ of $\Z^d$ and an isometry $v$ satisfying
\begin{itemize}
\item[(E1)] $vu_m=u_{Bm}$, 
\smallskip
\item[(E2)] $v^*u_mv=\begin{cases} u_{B^{-1}m}&\text{if $m\in B\Z^d$}\\
0&\text{otherwise, and}\end{cases}$
\smallskip
\item[(E3)] $1=\sum_{m\in\Sigma} (u_mv)(u_mv)^*$;
\end{itemize}
we then have
\[
C(\T^d)\rtimes_{\alpha_{B^t},L}\N=\clsp\{u_mv^kv^{*l}u_n^*:k,l\in \N\text{ and }m,n\in \Z^d\}.
\]
When we view $C(\T^d)\rtimes_{\alpha_{B^t},L}\N$ as a Cuntz-Pimsner algebra $(\O(M_L),j_{M_L},j_{C(\T^d)})$ as in Remark~\ref{idcpCP}, $v=j_{M_L}(q(1))=t$ and $u$ is the representation $m\mapsto j_{C(\T^d)}(\gamma_m)$, where $\gamma_m(z):=z^m$. Since $\alpha_{B^t}(\gamma_m)=\gamma_{Bm}$, Corollary~\ref{Ex=St2} gives an endomorphism $\beta$ of
\[
\big(C(\T^d)\rtimes_{\alpha_{B^t},L}\N\big)^{\hat\alpha_{B^t}}=\clsp\{u_mv^iv^{*i}u_n^*:i\in \N\text{ and }m,n\in \Z^d\}
\]
satisfying
\begin{equation}\label{defbeta4}
\beta(u_mv^iv^{*i}u_n^*)=u_{Bm}v^{i+1}v^{*(i+1)}u_{Bn}^*.
\end{equation}

Now that we have the formula for $\beta$, we can prove directly that there is such an endomorphism. To see this, we choose a set $\Sigma$ of coset representatives for $\Z^d/B\Z^d$, and recall from \cite[Proposition~5.5(b)]{LRR} that for each $i$ and 
\[
\Sigma_i:=\{\mu_1+B\mu_2+\cdots B^{i-1}\mu_i:\mu\in\Sigma^i\},
\]
$\{u_mv^iv^{*i}u_n^*:m,n\in\Sigma_i\}$ is a set of nonzero matrix units. Thus there is a homomorphism $\zeta_i:M_{\Sigma_i}(\C)\to M_{\Sigma_{i+1}}(\C)$ such that $\zeta_i(u_mv^iv^{*i}u_n^*)=u_{Bm}v^{i+1}v^{*(i+1)}u_{Bn}^*$. Let $\delta$ denote the universal representation of $B^1\Z^d$ in $C^*(B^{i+1}\Z^d)$. Then the unitary representation $B^im\mapsto \delta_{B^{i+1}m}$ induces a homomorphism $\eta_i:C^*(B^i\Z^d)\to C^*(B^{i+1}\Z^d)$. When we identify $C_i:=\clsp\{u_mv^iv^{*i}u_n^*\}$ with $M_{\Sigma_i}(\C)\otimes C^*(B^i\Z^d)$ as in \cite[Proposition~5.5(c)]{LRR}, we get homomorphisms 
\[
\beta_i:=\zeta_i\otimes \eta_i:C_i=M_{\Sigma_i}(\C)\otimes C^*(B^i\Z^d)\to C_{i+1}=M_{\Sigma_{i+1}}(\C)\otimes C^*(B^{i+1}\Z^d) 
\]
satisfying \eqref{defbeta4}. The Cuntz relation (E3) implies that $\beta_{i+1}|_{C_i}=\beta_i$, and hence the $\beta_i$ combine to give a homomorphism $\beta:\bigcup_{i=1}^\infty C_i\to \bigcup_{i=1}^\infty C_i$ satisfying \eqref{defbeta4}; since the homomorphisms $\beta_i$ are norm-decreasing, $\beta$ extends to an endomorphism of $\O(M_L)^\gamma=\overline{\bigcup_{i=1}^\infty C_i}$.
\end{example}


\section{Cuntz-Pimsner algebras}\label{cp-algs}  Let $(A, \alpha, L)$ be an Exel system as in \S\ref{mainthm}, and let $I$ be an ideal in $A$. Following \cite[Definition~4.1]{BR}, we say that $L$  is \emph{almost faithful on $I$} if 
\[a\in I\text{\ and\ }L(b^*a^*ab)=0\text{\ for all $b\in A$ imply }a=0. 
\]
Since 
\begin{equation}\label{almostfaithful}
L(b^*a^*ab)=\langle q(ab),q(ab)\rangle=\langle\phi(a)(q(b)),\phi(a)(q(b))\rangle,
\end{equation}
$L$ is almost faithful on $I$ if and only if $\phi|_I:I\to \L(M_L)$ is injective. In this section, we suppose that $L$ is almost faithful on $\phi^{-1}(\K(M_L))$. Then \cite[Proposition~2.1]{MS} implies that the canonical map $k_A:A\to \O(M_L)$ is injective. This will allow us to use \cite[Corollary~4.9]{FMR} to realise the core $\O(M_L)^\gamma$ as a direct limit. 

We denote the identity operator on $M_L^{\otimes i}$ by $1_i$, and write $\K(M_L^{\otimes j})\otimes_A 1_{i-j}$ for the image of $\K(M_L^{\otimes j})$ under the map $T\mapsto T\otimes_A 1_{i-j}$. Then, following \cite[\S4]{FMR}, we define
\[
C_i=(A\otimes_A1_i)+(\K(M_L)\otimes_A1_{i-1})+\cdots+\K(M_L^{\otimes i}),
\]
which is a $C^*$-subalgebra of $\L(M_L^{\otimes i})$. We define $\phi_i:C_i\to C_{i+1}$ by $\phi_i(T)=T\otimes_A 1_i$, and define $(C_\infty,\iota^i):=\varinjlim (C_i,\phi_i)$. Since $k_A$ is injective, we can now  apply \cite[Corollary~4.9]{FMR} to the Cuntz-Pimsner covariant representation $(k_M,k_A)$, and deduce that there is an isomorphism $\kappa$ of $C_\infty$ onto the core $\O(M_L)^\gamma$ such that
\begin{equation}\label{defkappa}
\kappa(\iota^i(T\otimes_A1_{i-j}))=(k_M^{\otimes j},k_A)^{(1)}(T)\ \text{ for $T\in \K(M_L^{\otimes j})$ and $i\geq j$}
\end{equation}
(the notation in \cite{FMR} suppresses the maps $\iota^i$ ).

To describe the endomorphism $\beta:=\kappa^{-1}\circ\alpha'\circ\kappa$ of $C_\infty$, we need some notation.

\begin{lemma}\label{defU}
The map $U:A\to M_L$ defined by $U(a)=q(\alpha(a))$ is an adjointable isometry such that $U^*(q(a))=L(a)$ for $a\in A$.
\end{lemma}

\begin{proof}
The calculation
\begin{equation}
\langle U(a)\,,\,U(b)\rangle=\langle q(\alpha(a))\,,\,q(\alpha(b))\rangle=L(\alpha(a)^*\alpha(b))=L(\alpha(a^*b))=a^*b=\langle a\,,\,b\rangle
\end{equation}
shows that  $U$ is inner-product preserving. We next note that 
\begin{equation}\label{compU*}
\langle U(a)\,,\,q(b)\rangle=L(\alpha(a)^*b)=a^*L(b)=\langle a\,,\,L(b)\rangle.
\end{equation}
Equation~\eqref{compU*} implies that 
\[
\|a^*L(b)\|=\|\langle U(a)\,,\,q(b)\rangle\|\leq \|U(a)\|\,\|q(b)\|=\|a\|\,\|q(b)\|;
\]
thus $\|L(b)\|\leq\|q(b)\|$, and there is a well-defined bounded linear map $T:M_L\to A$ such that $T(q(b))=L(b)$. Now \eqref{compU*}
shows that $U$ is adjointable with adjoint $T$.
\end{proof}

\begin{cor}\label{defUi}
Define maps $U_i:M_{L}^{\otimes i}\to M_L^{\otimes(i+1)}$ by identifying $M_{L}^{\otimes i}$ with $A\otimes_A M_L^{\otimes i}$ and taking
\[
U_i:=U\otimes_A 1_i:M^{\otimes i}=A\otimes_A M_L^{\otimes i}\to M_L\otimes_A M_L^{\otimes i}=M_L^{\otimes(i+1)}.
\]
Then each $U_i$ is an adjointable isometry, and
\begin{enumerate}
\setlength\itemindent{-16pt} 
\item $U_i(a\cdot m)=q(\alpha(a))\otimes_A m$;
\item $U_i^*(q(a)\otimes_Am)=L(a)\cdot m$;
\item $U_{i+1}=U_i\otimes_A1$ and $U_{i+1}^*=U_i^*\otimes_A 1$.
\end{enumerate}
\end{cor}

With the notation of Corollary~\ref{defUi}, we can now describe  the endomorphism on $C_\infty$.

\begin{thm}\label{CP=St}
Suppose that $(A,\alpha,L)$ is an Exel system such that  $L$ is almost faithful on $\phi^{-1}(\K(M_L))$. Then there is an endomorphism $\beta$ of $C_\infty:=\varinjlim(C_i,\phi_i)$ such that 
\begin{equation}
\beta(\iota^i(T))=\iota^{i+1}(U_iTU_i^*)\ \text{ for $T\in C_i$,}\label{defbeta2}
\end{equation}
and $\beta$ is extendible and injective with range $\overline{\beta}(1)C_\infty\overline{\beta}(1)$. Let $\kappa$ be the isomorphism of $C_\infty$ onto $\O(M_L)^\gamma$ satisfying \eqref{defkappa}, and let $V$ be the isometry in $M(\O(M_L))$ such that $k_M(q(a))=k_A(a)V$ for $a\in A$ (as given by Theorem~\ref{relCPalg=St}). Then $(\kappa,V)$ is a Stacey-covariant representation of $(C_\infty,\beta)$ in $\O(M_L)$, and $\kappa\times V$ is an isomorphism of the Stacey crossed product $C_\infty\times_\beta\N$ onto $\O(M_L)$.
\end{thm}

\begin{proof}
We define $\beta:=\kappa^{-1}\circ\alpha'\circ \kappa$, where $\alpha'$ is the endomorphism from Theorem~\ref{relCPalg=St}. Let $a\cdot m,b\cdot n\in M_L^{\otimes i}$. Then $\kappa\circ \iota^i(\Theta_{a\cdot m,b\cdot n})=k_M^{\otimes i}(a\cdot m)k_M^{\otimes i}(b\cdot n)^*$, and  
\begin{align*}
\kappa\circ\beta\circ\iota^i(\Theta_{a\cdot m,b\cdot n})&=\alpha'\circ\kappa\circ\iota^i(\Theta_{a\cdot m,b\cdot n})\\
&=k_M^{\otimes (i+1)}(q(\alpha(a))\otimes_Am)k_M^{\otimes (i+1)}(q(\alpha(b))\otimes_An)^*\quad\quad\text{(using \eqref{defalpha'})}\\
&=\kappa\circ\iota^{i+1}(\Theta_{q(\alpha(a))\otimes_Am,q(\alpha(b))\otimes_A n})\\
&=\kappa\circ\iota^{i+1}(\Theta_{U_i(a\cdot m),U_i(b\cdot n)})\\
&=\kappa\circ\iota^{i+1}(U_i\Theta_{a\cdot m,b\cdot n}U_i^*).
\end{align*}
This gives \eqref{defbeta2} for $T\in\K(M_L^{\otimes i})$. For $j<i$ and $S\in \K(M_L^{\otimes j})$, we have
\begin{align*}
\beta(\iota^i(S\otimes_A1_{i-j}))&=\beta(\iota^j(S))=\iota^{j+1}(U_jSU_j^*)\\
&=\iota^{i+1}((U_j\otimes_A1_{i-j})(S\otimes_A1_{i-j})(U_j^*\otimes_A1_{i-j}))\\
&=\iota^{i+1}(U_i(S\otimes_A1_{i-j})U^*_i),
\end{align*}
and adding over $j$ gives \eqref{defbeta2} for arbitrary $T$ in $C_i$.
The theorem now follows from Theorem~\ref{relCPalg=St}.
\end{proof}

\begin{remark}
Equation~\eqref{almostfaithful} implies that  $\phi$ is injective if and only if $L$ is almost faithful on $A$.  For a classical system associated to a surjective local homeomorphism $\sigma:X\to X$, the transfer operator $L$ that averages over inverse images is always faithful. Other transfer operators for other surjections $\sigma:X\to X$ need not be faithful--- see \cite[Example~4.7]{BR}. The canonical transfer operators for corner endomorphisms, on the other hand, are never faithful but, as we shall see in \S\ref{sec:St=Ex}, they are often almost faithful. 
\end{remark}

\subsection*{Connections with a construction of Exel}
In \cite[Theorem~6.5]{E2}, Exel shows that if $A$ is unital,  $\alpha$ is injective and unital, and  there is a faithful conditional expectation $E$ of $A$ onto $\alpha(A)$, then his crossed product $A\rtimes_{\alpha,\alpha^{-1}\circ E}\N$ is isomorphic to a Stacey crossed product $\check\AA\times_{\beta'}\N$. The $C^*$-algebra $\check\AA$ is by definition a subalgebra of the $C^*$-algebraic direct limit of a sequence of algebras of the form $\L(M_i)$ for certain Hilbert modules $M_i$. Exel's hypotheses on $\alpha$ imply that $L:=\alpha^{-1}\circ E$ is a faithful transfer operator for $\alpha$ satisfying $L(1)=1$, and hence that $\phi:A\to \L(M_L)$ is injective. Thus the Exel crossed product $A\rtimes_{\alpha,L}\N$ is the Cuntz-Pimsner algebra $\O(M_L)$, and Theorem~\ref{CP=St} gives an isomorphism of $C_\infty\times_{\beta}\N$ onto $A\rtimes_{\alpha,L}\N$. It is natural to ask whether Exel's system $(\check\AA,\beta')$ is the same as the system $(C_\infty,\beta)$ appearing in Theorem~\ref{CP=St}.

Exel's module $M_i$ is the Hilbert module over $\alpha^i(A)$ associated to the expectation $\alpha^i\circ L^i$ of $A$ onto $\alpha^i(A)$ (which he denotes by $\EE_i$), 
and hence is a completion of a copy $q_i(A)$ of $A$. By restricting the action we can view $M_i$ as a module over $\alpha^{i+1}(A)$, and this induces a linear map $j_i:M_i\to M_{i+1}$; Lemma~4.7 of \cite{E2} says that there is a homomorphism $\phi_i:\L(M_i)\to \L(M_{i+1})$ characterised by $\phi_i(T)\circ j_i=j_i\circ T$. (Exel writes the maps $j_i$ as inclusions.) Since $L^i$ is a transfer operator for $\alpha^i$, $\alpha^i\circ L^i$ extends to a self-adjoint projection $\check e_i$ in $\L(M_i)$. 
Exel's $C^*$-algebra $\check\AA$ is the $C^*$-subalgebra of $\varinjlim (\L(M_i),\phi_i)$ generated by the images of $A=\L(M_0)$ and $\{\check e_i:i\in\N\}$. Propositions 4.2 and 4.3 of \cite{E2} say that $\alpha$ and $L$ extend to isometric linear maps $\alpha_i:M_i\to M_{i+1}$ and $L_i:M_{i+1}\to M_i$, Proposition~4.6 of \cite{E2} says that the maps $\beta_i':T\mapsto \alpha_i\circ T\circ L_i$ are injective homomorphisms of $\L(M_i)$ into $\L(M_{i+1})$, and Proposition 4.10 of \cite{E2} says that they induce an endomorphism $\beta'$ of $\varinjlim (\L(M_i),\phi_i)$ which leaves $\check\AA$ invariant and satisfies $\beta'(\check e_i)=\check e_{i+1}$ for $i\geq 0$.

To compare our construction with that of \cite[\S4]{E2}, we use the maps 
\[
V_i:q(a_1)\otimes_A\cdots\otimes_Aq(a_i)\mapsto
q_i\big(a_1\alpha(a_2)\alpha^2(a_3)\cdots \alpha^{i-1}(a_i)\big);
\]
the pairs $(V_i,\alpha^i)$ then form compatible isomorphisms of $(M_L^{\otimes i},A)$ onto $(M_i,\alpha^i(A))$, and induce isomorphisms $\theta_i$ of $\L(M_L^{\otimes i})$ onto $\L(M_i)$. One quickly checks that the isometries $U_i$ of Corollary~\ref{defUi} satisfy
\begin{align*}
V_{i+1}U_iV_i^{-1}(q_i(a))&=V_{i+1}U_i(q(a)\otimes_A1\cdots\otimes_A1\\
&=V_{i+1}(q(1)\otimes_Aq(a)\otimes_A\cdots \otimes q(1))\\
&=q_{i+1}(\alpha(a))=\alpha_i(q_i(a)),
\end{align*}
and similiarly $V_iU_i^*V_{i+1}(q_{i+1}(a))=q_i(L(a))=L_i(q_{i+1}(a))$. Thus our endomorphism $\Ad U_i$ is carried into Exel's $\beta_i'$. The isomorphisms $\theta_i$ combine to give an injection of our direct limit $C_\infty=\varinjlim(C_i,\phi_i)$ into $\varinjlim (\L(M_i),\phi_i)$, and since $t^it^{*i}=\beta^i(1)$ is carried into $(\beta')^i(1)=\check e_i$, the formula \eqref{presentfpas} implies that the range of this injection is $\clsp\{a\check e_ib:a,b\in A\}$, which by \cite[Proposition~4.9]{E2} is precisely $\check\AA$. So $(C_\infty,\beta)$ is indeed isomorphic to $(\check\AA,\beta')$, and Theorem~\ref{CP=St} extends \cite[Theorem~6.5]{E2}.


\section{Stacey crossed products as Exel crossed products}\label{sec:St=Ex}

Stacey crossed products are particularly useful for studying \emph{corner endomorphisms} which map a $C^*$-algebra onto a corner $pAp$ (see \cite{P, BKR,LR2}, for example).
In his original paper on the subject, Exel proved that if $B$ is unital and $\beta\in \End B$ is an injective corner endomorphism, then the Stacey crossed product $B\times_\beta \N$ is an Exel crossed product $B\rtimes_{\beta,K}\N$ for the transfer operator $K:b\mapsto \beta^{-1}(\beta(1)b\beta(1))$ \cite[Theorem~4.7]{E}. Since the endomorphisms $\alpha'$ and $\beta$ appearing in our main theorems are corner endomorphisms, we are interested in a version of this result for nonunital algebras.

\begin{prop}\label{cornerendo}
Suppose that $\beta$ is an extendible endomorphism of a $C^*$-algebra $B$ such that $\beta$ is injective and $\beta(B)=\overline{\beta}(1)B\overline{\beta}(1)$. Then $K:b\mapsto \beta^{-1}(\overline{\beta}(1)b\overline{\beta}(1))$ is a transfer operator for $\beta$, and $K$ is almost faithful on $\overline{B\beta(B)B}=\overline{B\overline{\beta}(1)B}$.
\end{prop}

\begin{proof}
It is easy to see that $K$ is a positive linear map with norm $1$, and a very quick calculation shows that it satisfies $K(\beta(b)c)=bK(c)$ for $b,c\in B$. Since $\beta$ is injective, the extension $\overline\beta$ is an isomorphism of $M(B)$ onto $M(\overline{\beta}(1)B\overline{\beta}(1))$. Multipliers of the form $\overline{\beta}(1)m\overline{\beta}(1)$ in $M(B)$ multiply the corner $\overline{\beta}(1)B\overline{\beta}(1)$, so we can define $\overline{K}:M(B)\to M(B)$ by $\overline{K}(m)=(\overline{\beta})^{-1}(\overline{\beta}(1)m\overline{\beta}(1))$, and this has the required property $K(\beta(b)m)=b\overline{K}(m)$ for $b\in B$ and $m\in M(B)$.

To see that $K$ is almost faithful, we suppose that $b\in\overline{B\beta(B)B}$  satisfies $K(c^*b^*bc)=0$ for all $c\in B$, and  prove that $b=0$. We have
\begin{align*}
K(c^*b^*bc)=0 \text{ for all $c\in B$}
&\Longrightarrow \overline{\beta}(1)(c^*b^*bc)\overline{\beta}(1)=0\text{ for all $c\in B$}\\
&\Longrightarrow bc\overline{\beta}(1)=0\text{ for all $c\in B$}\\
&\Longrightarrow bc\overline{\beta}(1)d=0\text{ for all $c,d\in B$}\\
&\Longrightarrow b(\overline{B\beta(B)B})=0\\
&\Longrightarrow bb^*=0,
\end{align*}
which implies $b=0$. Thus $K$ is almost faithful. An approximate identity argument shows that $\overline{B\overline{\beta}(1)B}\subset \overline{B\beta(B)B}$, and the reverse inclusion holds because the range of $\beta$ is $\overline{\beta}(1)B\overline{\beta}(1)$.
\end{proof}

We can now give a generalisation of \cite[Theorem~4.7]{E}.

\begin{thm}\label{extendExel}
Suppose that $\beta$ is an extendible endomorphism of a $C^*$-algebra $B$ such that $\beta$ is injective and has range $\overline{\beta}(1)B\overline{\beta}(1)$, and $(B,\beta,K)$ is the Exel system of Proposition~\ref{cornerendo}. Let $(k_M,k_B)$ be the universal  representation of $M_K$ in $B\rtimes_{\beta,K}\N:=\O(K_\beta,M_K)$, and let $V$  be the isometry in $M(\O(K_\beta,M_K))$ such that $k_M(q(b))=k_B(b)V$ (as given by Theorem~\ref{relCPalg=St}). Then $(k_B,V)$ is a Stacey-covariant representation of $(B,\beta)$, and $k_B\times V$ is an isomorphism of the Stacey crossed product $B\times_\beta\N$ onto $B\rtimes_{\beta,K}\N$.
\end{thm}

The proof of Theorem~\ref{extendExel} uses a simple lemma which we will need again.

\begin{lemma}\label{comprank1}
Suppose that $(B,\beta,K)$ are as in Proposition~\ref{cornerendo}. Then for every $b,c\in B$, we have $\Theta_{q(b),q(c)}=\phi(b\overline{\beta}(1)c^*)$. \end{lemma}

\begin{proof} 
We take  $d\in B$ and compute:
\begin{equation}\label{eq-temp}
\Theta_{q(b),q(c)}(q(d))=q(b)\cdot\langle q(c)\,,\, q(d)\rangle=q(b\beta( K(c^*d)))=q(b\overline{\beta}(1)(c^*d)\overline{\beta}(1)).
\end{equation}
A direct calculation shows that $\|q(b\overline\beta(1)a\overline\beta(1))-q(b\overline\beta(1)a)\|=0$ for all $a\in B$, so \eqref{eq-temp} gives
\[\Theta_{q(b),q(c)}(q(d))
=q(b\overline{\beta}(1)(c^*d))
=\phi(b\overline{\beta}(1)c^*)(q(d)).\qedhere
\]
\end{proof}

\begin{proof}[Proof of Theorem~\ref{extendExel}] We choose an approximate identity $\{e_\lambda\}$ for $B$, and then Lemma~\ref{comprank1} implies that $\Theta_{q(\beta(b)),q(e_\lambda)}$ converges in norm in $\L(M_K)$ to $\phi(\beta(b))$. This implies, first, that $\beta(b)$ belongs to $K_\beta:=\overline{B\beta(B)B}\cap \phi^{-1}(\K(M_K))$, and, second, that 
\begin{equation}\label{otoh}
(k_M,k_B)^{(1)}(\Theta_{q(\beta(b)),q(e_\lambda)})\to (k_M,k_B)^{(1)}(\phi(\beta(b)))=k_B(\beta(b)).
\end{equation} 
But
\[
(k_M,k_B)^{(1)}(\Theta_{q(\beta(b)),q(e_\lambda)})=k_M(q(\beta(b)))k_M(q(e_\lambda))^*=k_B(\beta(b))VV^*k_B(e_\lambda),
\]
and $k_B$ is nondegenerate by \cite[Corollary~3.5]{BRV}, so 
\begin{equation}\label{oto}
(k_M,k_B)^{(1)}(\Theta_{q(\beta(b)),q(e_\lambda)})\to k_B(\beta(b))VV^*=Vk_B(b)V^*.
\end{equation}
Together, \eqref{otoh} and \eqref{oto} imply that $k_B(\beta(b))=Vk_B(b)V^*$, which is Stacey covariance.

The induced homomorphism $k_B\times V:B\times_\beta \N\to \O(K_\beta, M_K)$ is equivariant for the dual action $\hat\beta$ and the gauge action $\gamma$.  By Proposition~\ref{cornerendo}, the transfer operator $K$ is almost faithful on $K_\beta$, and so Theorem~4.3 of \cite{BRV} implies that $k_B:B\to \O(K_\beta, M_K)$ is injective. Thus the dual-invariant uniqueness theorem (Proposition~\ref{diut}) implies that $k_B\times V$ is injective, and since its range contains all the generators $k_B(B)\cup k_B(B)V=k_B(B)\cup k_M(q(B))$, it is an isomorphism.

Finally, recall that we are identifying $\O(K_\beta, M_K)$ with $B\rtimes_{\beta, K}\N$ using \cite[Theorem~4.1]{BRV} (see Remark~\ref{idcpCP}).
\end{proof} 

Theorem~\ref{extendExel} applies to the endomorphism $\alpha'$ of Theorem~\ref{relCPalg=St} and hence to the endomorphism $\beta$ of Corollary~\ref{Ex=St2}. Together, Theorem~\ref{extendExel} and Corollary~\ref{Ex=St2} give an alternative description of every Exel crossed product as an Exel crossed product of a larger algebra:

\begin{cor}\label{St=Ex}
Suppose that $(A,\alpha,L)$ is an Exel system. Let $\beta$ be the endomorphism of $B:=(A\rtimes_{\alpha,L}\N)^{\hat\alpha}$ described in Corollary~\ref{Ex=St2}, and let $K$ be the transfer operator for $(B, \beta)$ described in Proposition~\ref{cornerendo}.  Let $V$ be the isometry in $M(B\rtimes_{\beta,K}\N)$ such that $k_{M_K}(q(b))=k_B(b)V$ for $b\in B$. Then there is an isomorphism $\Upsilon$ of $A\rtimes_{\alpha,L}\N$ onto $B\rtimes_{\beta,K}\N$ such that $\Upsilon\circ k_A=k_B\circ k_A$ and $\Upsilon(t)=V$.
\end{cor}

\begin{proof}
We let $\Upsilon_1:B\times_\beta \N\to A\rtimes_{\alpha,L}\N$ be the isomorphism of Corollary~\ref{Ex=St2}, let $\Upsilon_2:B\times_\beta\N\to B\rtimes_{\beta,K}\N$ be the isomorphism of Theorem~\ref{extendExel}, and then 
$\Upsilon=\Upsilon_2\circ\Upsilon_1^{-1}$ has the required properties.
\end{proof}

The Stacey system $(C_\infty,\beta)$ appearing in Theorem~\ref{CP=St} also involves a corner endomorphism, and we can apply Theorem~\ref{extendExel} to this system. We show below that the ideal $K_\beta$ in $C_\infty$ is $\phi^{-1}(\K(M_K))$, and hence the relative Cuntz-Pimsner algebra $\O(K_\beta,M_K)$ is the Cuntz-Pimsner algebra $\O(M_K)$.  For the proof we need the following general lemma.

\begin{lemma}\label{cornervsK}
Suppose that $(B,\beta,K)$ is the Exel system described in Proposition~\ref{cornerendo}. Then 
\begin{enumerate}\setlength\itemindent{-16pt}
\item $\K(M_K)=\phi\big(\overline{B\beta(B)B}\big)$, and 
\item $\overline{B\beta(B)B}\cap \ker\phi=\{0\}$. 
\end{enumerate}
If the ideal $\overline{B\beta(B)B}$ is essential in $B$, then $\phi$ is injective and 
\[
K_\beta=\overline{B\beta(B)B}=\phi^{-1}(\K(M_K)).
\]
\end{lemma} 

\begin{proof}
The equality in (a) follows from Lemma~\ref{comprank1} and the identity $\overline{B\beta(B)B}=\overline{B\overline{\beta}(1)B}$ from Proposition~\ref{cornerendo}. For (b), we let $b\in B$, and then
\begin{align*}
\phi(b)=0&\Longleftrightarrow q(bc)=\phi(b)(q(c))=0\text{ for all $c\in B$}\\
&\Longleftrightarrow K(c^*b^*bc)=0\text{ for all $c\in B$}\\
&\Longleftrightarrow \overline{\beta}(1)(c^*b^*bc)\overline{\beta}(1)=0\text{ for all $c\in B$}\\
&\Longleftrightarrow bc\overline{\beta}(1)=0\text{ for all $c\in B$}\\
&\Longleftrightarrow b\big(B\overline{\beta}(1)B\big)=\{0\};
\end{align*}
since the intersection of two ideals $I$ and $J$ is $\clsp\{ij:i\in I,\;j\in J\}$, this implies (b). If $\overline{B\beta(B)B}$ is essential, then it must meet every non-zero ideal non-trivially, so (b) implies that $\ker\phi=\{0\}$. Now applying $\phi^{-1}$ to both sides of (a) gives the last assertion.
\end{proof}

\begin{prop}\label{CP=CP}
Suppose that $(A,\alpha,L)$ is an Exel system such that $L$ is almost faithful on $\phi^{-1}(\K(M_L))$. Let $(C_\infty,\beta)$ be the Stacey system constructed in Theorem~\ref{CP=St} and let $K:b\mapsto \beta^{-1}\big(\overline{\beta}(1)b\overline{\beta}(1)\big)$ be the transfer operator for $(C_\infty, \beta)$ from Proposition~\ref{cornerendo}.
\begin{enumerate}
\item\label{CP=CPa} The ideal $J:=\overline{C_\infty\beta(C_\infty)C_\infty}$ is essential in $C_\infty$.
\item\label{CP=CPb} $K_\beta=J=\phi^{-1}(\K(M_K))$.
\item\label{CP=CPc} Let $V$ be the isometry in $M(\O(M_K))$ such that $k_M(q(c))=k_{C_\infty}(c)V$ for $c\in C_\infty$. Then $(k_{C_\infty},V)$ is a Stacey-covariant representation of $(C_\infty,\beta)$ in $\O(M_K)$, and $k_{C_\infty}\times V$ is an isomorphism of the Stacey crossed product $C_\infty\times_\beta \N$ onto the Exel crossed product $C_\infty\rtimes_{\beta, K}\N=\O(M_K)$. 
\end{enumerate}
\end{prop}

\begin{proof} Once we have established (\ref{CP=CPa}), (\ref{CP=CPb}) will follow from Lemma~\ref{cornervsK}, and then (\ref{CP=CPc}) will follow from Theorem~\ref{extendExel}. 

To prove (\ref{CP=CPa}), we first need to get our hands on some elements of the ideal $J$.
Recall that the system $(C_\infty,\beta)$ is pulled back from $(\O(M_L)^\gamma,\alpha')$ along the isomorphism $\kappa$ described in \eqref{defkappa}. The elements $V^iV^{*i}=V^iV^{*i}VV^*$ with $i\geq 1$ belong to the ideal in $\O(M_L)^\gamma$ generated by $VV^*=\overline{\alpha'}(1)=\overline{\kappa}(\overline{\beta}(1))$, and hence so do all elements of the form
\begin{align*}
k_M^{\otimes i}\big((q(a_1)&\otimes_A\cdots\otimes_A q(a_i)\big)k_M^{\otimes i}\big(q(b_1)\otimes_A\cdots\otimes_A q(b_i)\big)^*\\&=k_A\big(a_1\alpha(a_2)\alpha^2(a_3)\cdots \alpha^{i-1}(a_i)\big)V^iV^{*i}k_A\big(b_1\alpha(b_2)\alpha^2(b_3)\cdots \alpha^{i-1}(b_i)\big)^*
\end{align*}
(see the calculation in \eqref{simplifyinEcp}). For $m=q(a_1)\otimes_A\cdots\otimes_A q(a_i)$ and $n=q(b_1)\otimes_A\cdots\otimes_A q(b_i)$ we have $k_M^{\otimes i}(m)k_M^{\otimes i}(n)^*=\kappa(\iota^i(\Theta_{m,n}))$. So $\iota^i(\Theta_{m,n})$ belongs to the ideal in $C_\infty$ generated by $\overline{\beta}(1)$, which is precisely $J$. Thus $\iota^i(\K(M_L^{\otimes i}))\subset J$ for every $i\geq 1$.

An ideal $J$ is essential if it has nonzero intersection with every nonzero ideal. We suppose that $I$ is an ideal in $C_\infty$ such that $IJ=I\cap J=\{0\}$, and aim to show that $I=\{0\}$. Since $C_\infty=\overline{\bigcup_{i=1}^\infty\iota^i(C_i)}$, we know from \cite[Lemma~1.3]{ALNR}, for example, that $I=\overline{\bigcup_{i=1}^\infty\iota^i(C_i)\cap I}$. So it suffices to prove that $\iota^i(C_i)\cap I=\{0\}$ for every $i\geq 1$. Suppose that $T\in \K(M_L^{\otimes i})$ for some $i\geq 1$ and that $\iota^i(T)$ belongs to $I$. Then the assumption $IJ=\{0\}$ and the inclusion in the previous paragraph imply that 
\begin{equation}\label{proveJess}
\iota^i(TS)=\iota^i(T)\iota^i(S)=0\ \text{ for every $S\in \K(M_L^{\otimes i})$.}
\end{equation}
Since $\L(M_L^{\otimes i})$ is the multiplier algebra of $\K(M_L^{\otimes i})$ (see \cite[Corollary~2.54]{tfb}, for example), $\K(M_L^{\otimes i})$ is essential in $\L(M_L^{\otimes i})$, and \eqref{proveJess} implies that $T=0$. Thus $\iota^i(C_i)\cap I=\{0\}$, as required, and $J$ is essential in $C_\infty$. As indicated at the beginning, this completes the proof of the proposition.
\end{proof}

\begin{remark}\label{rmkSt=CP}
When we combine Proposition~\ref{CP=CP} with Theorem~\ref{CP=St}, we obtain an isomorphism of $\O(M_L)$ onto $\O(M_K)=C_\infty\rtimes_{\beta,K}\N$.  So even though $\O(M_L)$ need not be the Exel crossed product of the orginal system $(A,\alpha,L)$, it can still be realised as an Exel crossed product. 
\end{remark}


\section{The example of Ionescu and Muhly}\label{sec-IM}

We now discuss the example of  \cite[Theorems~3.3 and 3.4]{IM} and \cite[Proposition~3.3]{D} which motivated this paper. Let $\sigma$ be a surjective local homeomorphism of a compact Hausdorff space $X$.  The \emph{Deaconu-Renault} groupoid (named after its use in \cite{D2} and a special case in  \cite[page~138]{ren})  is 
\[G:=\{(x,n,y)\in X\times\Z\times X:\sigma^k(x)=\sigma^l(y)\text{\ for some $k,l\in \N$ with $n=k-l$}\}\]
with unit space $G^{(0)}=X$, source and range maps $s((x,n,y))=y$, $r((x,n,y))=x$, product $(x,n,z)(z,k,y)=(x, n+k, y)$, and inverse $(x,n,y)^{-1}=(y,-n,x)$. Suppose that $U$, $V$ are open subsets of $X$, and $k,l$ are natural numbers such that $\sigma^k|_U$ and $\sigma^l|_V$ are homeomorphisms with $\sigma^k(U)=\sigma^l(V)$, and set
\[
Z(U,V, k,l)=\{(x,k-l, y)\in  G: x\in U, y\in V\}.
\]
It is shown in \cite[\S3]{ER} that the sets $Z(U,V, k,l)$ are a basis for a locally compact Hausdorff topology on $G$, and that $G$ is then an $r$-discrete groupoid for which the counting measures form a Haar system.  Since  $X=G^{(0)}$ is open in $G$, $C(X)=C(G^{(0)})$ embeds isometrically in  $C^*(G)$. 

As in \cite{IM}, the function  $S\in C_c(G)$ defined  by 
\begin{equation}\label{eqS} 
S(x,k, y)=\begin{cases} \frac{1}{\sqrt{|\sigma^{-1}(\sigma(x))|}} &\text{if $k=1$ and $y=\sigma(x)$}\\0&\text{otherwise,}\end{cases}
\end{equation} is an isometry in $C^*(G)$.  The discussion on page~110 of \cite{ren} show that there is an action $\gamma':\T\to \Aut C^*(G)$ such that $\gamma'_z(f)(x,k,y)=z^kf(x,k,y)$. The following is a restatement of \cite[Theorem~3.3]{IM}.

\begin{thm}\label{thmIM1}\textup{(Deaconu, Ionescu-Muhly)} \label{IMthm1}
Let $\alpha$ be the endomorphism  $f\mapsto f\circ\sigma$ of $C(X)$,
and let $L$ be the transfer operator for $(C(X),\alpha)$ defined by \[L(f)(x)=\frac{1}{|\sigma^{-1}(x)|}\sum_{\sigma(y)=x}f(y).\] 
Let $\iota$ be the identification of $C(X)$ with the isometric embedding of $C(G^{(0)})$ in $C^*(G)$ and  $S$ the isometry of \eqref{eqS}, and  define $\psi:M_L\to C^*(G)$  by $\psi(f)=\iota(f)S$.  Then $(\psi, \iota)$ is   a  Cuntz-Pimsner covariant representation of $M_L$, and $\psi\times\iota$ is a $\gamma$--$\gamma'$ equivariant isomorphism of $\O(M_L)$  onto $C^*(G)$.
\end{thm}

For $n,m\in\N$, set $R_{nn}:=\{(x,0,y)\in G: \sigma^n(x)=\sigma^n(y)\}$ and 
$R_\infty :=\bigcup_{n\in\N}R_{n}$, which is an open and closed subgroupoid of $G$. Define $c:G\to\Z$ by $c(x,n,y)=n$, and note that $c$ is a continuous homomorphism such that $R_\infty=c^{-1}(0)$. We can therefore deduce from \cite[Theorem~6.2]{KQR}, for example, that the inclusion map induces an isometric embedding of $C^*(R_\infty)$ into $C^*(G)$.  
We can now restate \cite[Theorem~3.4]{IM} as follows:

\begin{thm}\textup{(Ionescu-Muhly)}\label{IMthm2}  The inclusion $i:C_c(R_\infty)\to C_c(G)\subset C^*(G)$ extends to an isomorphism $i$ of $C^*(R_\infty)$ onto the fixed-point algebra $C^*(G)^{\gamma'}$.  Let $S$ be the isometry in \eqref{eqS}. Then $(C^*(R_\infty),\Ad S,\Ad S^*)$ is an Exel system.   Define $\rho: C_c(R_\infty)\to C^*(G)$ by $\rho(h)=i(h)S$. Then $(\rho,i)$ extends to a Cuntz-Pimsner covariant representation of $M_{\Ad S^*}$, and $\rho\times i$ is an isomorphism of $\O(M_{\Ad S^*})$ onto $C^*(G)$.
\end{thm}

Together Theorems~\ref{IMthm1} and \ref{IMthm2} give an isomorphism of $\O(M_L)$ onto $\O(M_{\Ad S^*})$. We can now recover this isomorphism from our results and Theorem~\ref{IMthm1}.

\begin{cor}\label{cor-IM}
Let $j$ be the inclusion of $C(X)=C(R_\infty^{(0)})$  in $C^*(R_\infty)$.  Then there is an isomorphism $\Upsilon$ of $C(X)\rtimes_{\alpha, L}\N=\O(M_L)$ onto $C^*(R_\infty)
\rtimes_{\Ad S, \Ad S^*}\N=\O(M_{\Ad S^*})$ such that $\Upsilon\circ k_{C(X)}=k_{C^*(R_\infty)}\circ j$ and $\Upsilon\circ k_{M_L}=k_{M_{\Ad S^*}}\circ j$.
\end{cor}

\begin{proof} We first claim that $C(X)$ acts on the left of $M_L$ by compact operators. To see this, choose a finite open covering $\{U_i\}$ of $X$ such that $\sigma$ is a homeomorphism on each $U_i$, let $\{g_i\}$ be a partition of unity subordinate to $\{U_i\}$, and set $\eta_i(x)=\sqrt{|\sigma^{-1}(\sigma(x))|g_i(x)}$.  Then  \cite[Proposition~8.2]{EV} implies that
$\phi(f)=\sum_i\Theta_{f\cdot \eta_i,\eta_i}$ for $f\in C(X)$, and hence $\phi(f)$ is compact, as claimed. Since $\alpha$ is unital,  $K_\alpha=\phi^{-1}(\K(M_L))$, and $C(X)\rtimes_{\alpha, L}\N=\O(K_\alpha,M_L)$ coincides with  $\O(M_L)$. 

Since the isomorphism $\psi\times \iota$ of Theorem~\ref{IMthm1} is $\gamma$--$\gamma'$ equivariant we can pull the endomorphism $\alpha'$ of $\O(M_L)^\gamma$ from Theorem~\ref{relCPalg=St} over to an endomorphism $\tau$ of  $C^*(R_\infty)$ by setting $i\circ \tau:=(\psi\times \iota)\circ \alpha'\circ(\psi\times \iota)^{-1}$. Let $K$  be the transfer operator of $(C^*(R_\infty), \tau)$ defined by $K(b)=\tau^{-1}(\overline{\tau}(1)b\overline{\tau}(1))$. 
Then Corollary~\ref{St=Ex} says there is an isomorphism  $\Upsilon$ of $C(X)\rtimes_{\alpha, L}\N=\O(M_L)$ onto $C^*(R_\infty)\rtimes_{\tau, K}\N$ with the stated properties.  Since $L$ is faithful,  Proposition~\ref{CP=CP}(\ref{CP=CPb}) implies that $K_\tau=\phi^{-1}(\K(M_K))$, which implies that $C^*(R_\infty)\rtimes_{\tau, K}\N=\O(M_K)$.
\footnote{This contradicts the claim on page~201 of \cite{IM} that $\overline{C^*(R_\infty)\tau(C^*(R_\infty))C^*(R_\infty)}$ is a proper ideal.} 

Finally, we need to check that $(\tau, K)$ coincides with the Exel system of \cite[Theorem~3.4]{IM}. By the formula (3) at the top of page~197 of \cite{IM} we need to check that  $(\tau, K)=(\Ad S, \Ad S^*)$.  For this, recall from  Theorem~\ref{relCPalg=St} that $\alpha'$ is $(\Ad V)|$ where $V=k_{M_L}(1)$; since $\iota\rtimes \psi(V)=\psi(1)=\iota(1)S=S$ it follows that  $\tau$ is $\Ad S$. Now  $K(b)=\tau^{-1}(SS^*bSS^*)=S^*bS$, that is, $K$ is $\Ad S^*$.  Thus $M_K=M_{{\Ad} S^*}$ as required.
\end{proof}

Proposition~\ref{CP=CP} now implies that $C^*(G)$ is isomorphic to the Stacey crossed product $C^*(R_\infty)\times_{\Ad S}\N$, as also noticed by Anantharaman-Delaroche \cite[\S1.3.4]{AD} and Deaconu   \cite[page~1782]{D2}.


\section{K-theory}

Let $(A, \alpha, L)$ be an Exel system such that $L$ is almost faithful on $\phi^{-1}(\K(M_L))$.   Recall from Theorem~~\ref{CP=St} that the core $\O(M_L)^\gamma$ is isomorphic to a direct limit $C_\infty:=\varinjlim(C_i,\phi_i)$ and that there is an endomorphism $\beta$ of $C_\infty$ such that the Stacey crossed product $C_\infty\times_\beta\N$ is isomorphic to $\O(M_L)$.  In this section we establish a $6$-term cyclic exact sequence of K-groups of $C_\infty$ and $\O(M_L)$. For this we need to know that the range of $\beta$ is a full corner in $C_\infty$, and we show this with the following more general lemma.

\begin{lemma}\label{lem-full}
Suppose that $(A,\alpha,L)$ is an Exel system such that $L$ is almost faithful on $\phi^{-1}(\K(M_L))$. Let $(C_\infty,\beta)$ be the Stacey system constructed in Theorem~\ref{CP=St} and let $K:b\mapsto \beta^{-1}\big(\overline{\beta}(1)b\overline{\beta}(1)\big)$ be the transfer operator for $(C_\infty, \beta)$ from Proposition~\ref{cornerendo}.
Define\footnote{It's not clear whether there is much point in doing this lemma in this generality. In the application to Theorem~\ref{prop6term} we are assuming that $\phi\subset\K(M_L)$ and that $A$ ia unital. But under these hypotheses, $\K(M_L^{\otimes i})=\L(M_L^{\otimes i})$ is unital, and so is $C_\infty$.}
\[
I_i:=(\K(M_L)\otimes_A1_{i-1})+\cdots+\K(M_L^{\otimes i}).
\]
Then $I_i$ is an ideal in $C_i$, $\varinjlim I_i$ is naturally isomorphic to an ideal $I_\infty:=\overline{\bigcup_{i=1}^\infty\iota^i(I_i)}$ in $C_\infty$, and this ideal is precisely $J=\overline{C_\infty\beta(C_\infty)C_\infty}$.
Moreover, the following are equivalent:
\begin{enumerate}\setlength\itemindent{-16pt}
\item\label{lem-fullbi} $A$ acts by compact operators on $M_L$;
\item\label{lem-fullbii}  $C_\infty$ acts by compact operators on $M_K$;
\item\label{lem-fullbiii} $J=C_\infty$;
\item\label{lem-fullbiv}  the range of $\beta$ is a full corner in $C_\infty$.
\end{enumerate}
\end{lemma}

\begin{proof} Recall that $(C_\infty, \iota^i)=\varinjlim(C_i,\phi_i)$ where  $\phi_i:C_i\to C_{i+1}$ is $\phi_i(T)=T\otimes_A 1_i$.

An induction argument shows that each $I_i$ is the sum of a $C^*$-subalgebra $I_{i-1}\otimes_A1$ and the ideal $\K(M_L^{\otimes i})$ of $\L(M_L^{\otimes i})$, hence is a $C^*$-algebra. For fixed $i\geq 1$, and for $0\leq j\leq i$ and $1\leq l\leq i$,  the product
\[
\big(\K(M_L^{\otimes j})\otimes 1_{i-j}\big)\big(\K(M_L^{\otimes l})\otimes 1_{i-l}\big)
\]
of typical summands of $C_i$ and $I_i$ is back in $I_i$, whence
 $I_i$ is an ideal in $C_i$. Thus $I_\infty:=\overline{\bigcup_{i=1}^\infty\iota^i(I_i)}$ is an ideal in $C_\infty=\overline{\bigcup_{i=1}^\infty\iota^i(C_i)}$. The maps $\iota^i$ induce an isomorphism of $\varinjlim I_i$ onto $I_\infty$. 
 
 Next observe that, since $\beta$ maps $\iota^i(\K(M_L^{\otimes i}))$ into 
$\iota^{i+1}(\K(M_L^{\otimes (i+1)}))$ (see \eqref{defbeta2}), the range of $\beta$ is contained in $I_\infty$, and hence $J\subset I_\infty$. But we showed in the proof above that $\iota^i(\K(M_L^{\otimes i}))$ is contained in $J$ for every $i\geq 1$, and hence so is
\[
\iota^i(I_i)=\iota^i\big((\K(M_L)\otimes_A1_{i-1}+\cdots+\K(M_L^{\otimes i})\big)=\iota^1(\K(M_L))+\cdots+\iota^i(K(M_L^{\otimes i})).
\]
Thus $I_\infty\subset J$, and we have proved $I_\infty=J$.  

Since $\phi_0:C_0\to C_1$ is the map $\phi_A:A\to \L(M_L)$, and since $\iota^1\circ\phi_0=\iota^0$,  $\iota^0:A\to C_\infty$ maps $\phi_A^{-1}(\K(M_L))$ into $\iota^1(\K(M_L))\subset I_\infty$.  Also, $T\otimes_A1_i\in \K(M_L^{\otimes(i+1)})$ implies $T\in \K(M_L)$  by \cite[Lemma 4.5]{FMR}, so $\phi_A^{-1}(\K(M_L))=(\iota^0)^{-1}(I_\infty)$. Thus $\iota^0$ induces an injection of $A/\phi_A^{-1}(\K(M_L))$ into $C_\infty/I_\infty$. Since every element of $\iota^i(C_i)$ has the form $\iota^0(a)+\iota^i(c)$ for some $a\in A$ and $c\in I_i$, $\iota$ has dense range, and must be surjective. Thus $\iota$ is an isomorphism of $A/\phi_A^{-1}(\K(M_L))$ onto $C_\infty/I_\infty$. From above, $I_\infty=J$, and from Proposition~\ref{CP=CP}, $J=\phi^{-1}(\K(M_K))$. So the existence of this isomorphism implies that $C_\infty$ acts by compact operators on $M_K$ if and only if $A$ acts by compact operators on $M_L$. This gives the equivalence of (\ref{lem-fullbi})  and (\ref{lem-fullbii}).  The equivalence of (\ref{lem-fullbii})  and (\ref{lem-fullbiii}) follows from Lemma~\ref{cornervsK}. The range of $\beta$ is $\overline{\beta}(1)C_\infty\overline{\beta}(1)$ by Theorem~\ref{CP=St},  and $C_\infty\overline{\beta}(1)C_\infty=J$ by Proposition~\ref{cornerendo}, so (\ref{lem-fullbiii})  and (\ref{lem-fullbiv}) are equivalent.
\end{proof}

\begin{prop}\label{prop6term}
Let $(A, \alpha, L)$ be an Exel system with $A$ unital. Also suppose that $A$ acts by compact operators on $M_L$ and that $L$ is almost faithful on $\phi^{-1}(\K(M_L))$=A.  Let $(C_\infty,\beta)$ be the Stacey system constructed in Theorem~\ref{CP=St} and let $K:b\mapsto \beta^{-1}\big({\beta}(1)b{\beta}(1)\big)$ be the transfer operator for $(C_\infty, \beta)$ from Proposition~\ref{cornerendo}. Assume that $C_\infty$ is unital, and let $k=k_{C_\infty}:C_\infty\to\O(M_K)$ be the canonical map. Then there exists a cyclic exact sequence
\begin{equation}\label{6term}
\xymatrix{
K_0(C_\infty)\ar[r]^{\beta_*-\id}&K_0(C_\infty)\ar[r]^-{k_*}&K_0(\O(M_K))\ar[d]_{\delta_0}
\\
K_1(\O(M_K))\ar[u]_{\delta_1}&K_1(C_\infty)\ar[l]_-{k_*}&K_1(C_\infty).\ar[l]_{\beta_*-\id}
}
\end{equation}
\end{prop}

\begin{proof}
By assumption $A$ acts by compact operators on $M_L$, and hence the range of $\beta$ is a full corner in $C_\infty$ by Lemma~\ref{lem-full}.  Hence Theorem~4.1 of \cite{P2} applies, and,  using the isomorphism $k_{C_\infty}\times V$ of the Stacey crossed product $C_\infty\times_\beta \N$ onto $\O(M_K)$ from Proposition~\ref{CP=CP}, gives \eqref{6term}.
\end{proof}
It may seem at first glance that Proposition~\ref{prop6term} has a lot of hypotheses. But these hypotheses hold in many situations,  and in particular when we start with the system  $(C(X),\alpha, L)$  associated to a local homeomorphism $\sigma:X\to X$ (see  \S\ref{sec-IM}). There  $\O(M_K)$ is isomorphic to the $C^*$-algebra of the Deaconu-Renault groupoid $G$ and $C_\infty$ is isomorphic to the $C^*$-algebra $C^*(R_\infty)$ of a subgroupoid.  Thus Proposition~\ref{prop6term} includes \cite[Theorem~2]{D2} and \cite[Proposition~9.1]{DS}.


\section{Graph algebras as crossed products}\label{ckex}
Let $E=(E^0, E^1, r,s)$ be a locally finite directed graph with no sources or sinks. We use the conventions of \cite{CBMS}.  Briefly, we think of $E^0$ as vertices, $E^1$ as edges, and $r,s:E^1\to E^0$ as describing the range and source of an edge.  Locally finite means that $E$ is both row-finite and column-finite, so that both $r^{-1}(v)$ and $s^{-1}(v)$ are finite for every $v\in E^0$.  We write $E^*$ for the set of finite paths $\mu=\mu_1\dots\mu_n$ satisfying $s(\mu_i)=r(\mu_{i+1})$ for $1\leq i\leq n-1$, and $|\mu|$ for the length $n$ of this path.  Similarly, we write $E^n$ for the set of paths of length $n$ and  $E^\infty$ for the set of infinite paths $\eta=\eta_1\eta_2\dots$.   We equip the path space $E^\infty$ with the product topology inherited from $\Pi_{n=1}^\infty E^1$, which is locally compact and Hausdorff  and has a basis  consisting of the cylinder sets $Z(\mu):=\{\eta\in E^\infty:\eta_i=\mu_i\ \text{for $1\leq i\leq |\mu|$}\}$ parametrised by $\mu\in E^*$.

Now consider the backward shift $\sigma$ on $E^\infty$ defined by $\sigma(\eta_1\eta_2\dots)=\eta_2\eta_3\dots$. Since $E$ has no sinks, $\sigma$ is surjective.  Since  $E$ is column-finite, $\sigma$ is a local homeomorphism which is proper in the sense that inverse images of compact sets are compact (see \cite[\S2.2]{BRV}).  Since $\sigma$ is proper, $\alpha:f\mapsto f\circ\sigma$ is a nondegenerate endomorphism of $C_0(E^\infty)$; since $E$ is column-finite, $\sigma^{-1}(\eta)$ is finite, and  
\[
L(f)(\eta)=\frac{1}{|\sigma^{-1}(\eta)|}\sum_{\sigma(\xi)=\eta} f(\xi)=\frac{1}{|s^{-1}(r(\eta))|}\sum_{s(e)=r(\eta)}f(e\eta)
\]
defines a transfer operator $L:C_0(E^\infty)\to C_0(E^\infty)$ for $\alpha$ by \cite[Lemma~2.2]{BRV}. Moreover, $L$ extends suitably to $M(C_0(E^\infty))$ so that $(C_0(E^\infty), \alpha, L)$ is an Exel system of the sort we've been considering. By \cite[Corollary~4.2]{BRV}, using a partition of unity argument similar to the one sketched in Corollary~\ref{cor-IM}, the action of $C_0(E^\infty)$ on $M_L$ is by compact operators. Hence $C_0(E^\infty)\rtimes_{\alpha, L}\N:=\O(K_\alpha, M_L)=\O(M_L)$.

A Cuntz-Krieger $E$-family in a $C^*$-algebra $B$ consists of a set $\{P_v:v\in E^0\}$ of mutually orthogonal projections and a family $\{T_e:e\in E^1\}$ of partial isometries such that $T^*_eT_e=P_{s(e)}$ for all $e\in E^1$ and $P_v=\sum_{r(e)=v}T_eT_e^*$ for all $v\in E^0$.  The $C^*$-algebra $C^*(E)$ of $E$  is the $C^*$-algebra universal for Cuntz-Krieger $E$-families; we write $\{t,p\}$ for the universal Cuntz-Krieger $E$-family that generates $C^*(E)$.
See \cite{CBMS} for more details.

 For $e\in E^1$, we define $m_e:=|s^{-1}(s(e))|^{1/2}q(\chi_{Z(e)})$. Since $E$ is locally finite with no sources or sinks,  we know from \cite[Theorem~5.1]{BRV} that
\[
T_e:=k_M(m_e)= |s^{-1}(s(e))|^{1/2}k_M(q(\chi_{Z(e)}))\text{\ and\ }P_v:=k_A(\chi_{Z(v)})
\]
form a Cuntz-Krieger $E$-family in $\O(M_L)$, and that $\pi_{T,P}:t_e\mapsto T_e$ and $p_v\mapsto P_v$ is an isomorphism of the graph algebra $C^*(E)$ onto $C_0(E^\infty)\rtimes_{\alpha,L}\N=\O(M_L)$. Note that $\pi_{T,P}$ is equivariant for the gauge actions.  Pulling over the endomorphism $\alpha'$ of Theorem~\ref{relCPalg=St} gives an endomorphism $\beta=\pi_{T,P}^{-1}\circ\alpha'\circ \pi_{T,P}$ of the core
\[C^*(E)^\gamma=\clsp\big\{t_\mu t_\nu^*:\mu,\nu\in E^*\text{ satisfy }|\mu|=|\nu|\big\}.\] 
We are going to compute $\beta$.

Let $\mu=\mu_1\dots\mu_i$ be a finite path  and $m_\mu:=m_{\mu_1}\otimes_A\cdots\otimes_A m_{\mu_i}$. Then
\[
T_\mu=T_{\mu_1}T_{\mu_2}\cdots T_{\mu_i}=k_M(m_{\mu_1})k_M(m_{\mu_2})\cdots k_M(m_{\mu_i})=k_M^{\otimes i}(m_\mu).
\]
To compute $\beta$ we first note that
 $m_\mu=\chi_{Z(r(\mu))}\cdot m_\mu$ and \[\alpha(\chi_{Z(\mu)})=\chi_{Z(\mu)}\circ\sigma=\sum_{s(e)=\mu}\chi_{Z(e)}.\]  So for paths $\mu$ and $\nu$ of length $i$ we have
\begin{align*}
\pi_{T,P}&(\beta(t_\mu t_\nu^*))
=\alpha'(T_\mu T_\nu^*)
=\alpha'(k_M^{\otimes i}(m_\mu)k_M^{\otimes i}(m_\nu)^*)\\
&=\alpha'\big(k_M^{\otimes i}(\chi_{Z(r(\mu))}\cdot m_\mu)k_M^{\otimes i}(\chi_{Z(r(\nu))}\cdot m_\nu)^*\big)\\
&=k_M^{\otimes i+1}(q(\alpha(\chi_{Z(r(\mu))})\otimes m_\mu)k_M^{\otimes i+1}(q(\alpha(\chi_{Z(r(\nu))})\otimes m_\nu)^*\quad  \text{(using \eqref{defalpha'})}\\
&=\sum_{s(e)=r(\mu),\;s(f)=r(\nu)}(|s^{-1}(s(e))s^{-1}(s(f))|)^{-1/2}k_M^{\otimes i+1}(m_e\otimes_Am_\mu)k_M^{\otimes i+1}(m_f\otimes_Am_\nu)^*\\
&=\sum_{s(e)=r(\mu),\;s(f)=r(\nu)}(|s^{-1}(s(e))s^{-1}(s(f))|)^{-1/2}k_M^{\otimes i+1}(m_{e\mu})k_M^{\otimes i+1}(m_{f\nu})^*.
\end{align*}
Now recall that $\alpha'$ is conjugation by an isometry   $V=\lim_\lambda k_M(q(e_\lambda))$, where $\{e_\lambda\}$ is  is any approximate identity of $C_0(E^\infty)$.   A quick calculation with, for example,  the approximate identity $\{e_F=\sum_{e\in F} \chi_{Z(e)}\}$ indexed by finite subsets $F$ of $E^1$, shows that  $W:=\pi_{T,P}^{-1}(V)=\sum_{e\in E^1}|s^{-1}(s(e))|^{-1/2}t_e$. So Theorem~\ref{relCPalg=St} gives:

\begin{prop}\label{CK=cp}
Suppose that $E$ is a locally-finite directed graph with no sources or sinks,  and $\{t_e,p_v\}$ is the universal Cuntz-Krieger $E$-family which generates $C^*(E)$. Then there is an endomorphism $\beta$ of the core $C^*(E)^\gamma$ such that
\begin{equation}\label{defbeta3}
\beta(t_\mu t_\nu^*)=\sum_{s(e)=r(\mu),\;s(f)=r(\nu)}\big(|s^{-1}(r(\mu))|\,|s^{-1}(r(\nu))|\big)^{-1/2}t_{e\mu} t_{f\nu}^*.
\end{equation}
The series
\begin{equation*}
\sum_{e\in E^1}|s^{-1}(s(e))|^{-1/2}t_e
\end{equation*}
converges strictly in $M(C^*(E))$ to an isometry $W$  satisfying $\beta(t_\mu t_\nu^*)=Wt_\mu t_\nu^* W^*$. If $\iota$ is the inclusion of the core $C^*(E)^\gamma$ in $C^*(E)$, then the associated  representation $\iota\times W$ of the Stacey crossed product $C^*(E)^\gamma\times_{\beta}\N$ is an isomorphism   onto $C^*(E)$. 
\end{prop}

\begin{remark}
Suppose that $E$ is the bouquet of $n$ loops on a single vertex, so that $C^*(E)$ is the Cuntz algebra $\OO_n$. Then the endomorphism $\beta$ is not the usual endomorphism $\alpha:\bigotimes_{k=1}^\infty a_k\mapsto e_{11}\otimes\big(\bigotimes_{k=2}^\infty a_{k-1}\big)$ of the UHF core $\OO_n^\gamma=\bigotimes_{k=1}^\infty M_n(\C)$ for which $\OO_n=\OO_n^\gamma\times_\alpha\N$, but it is closely related: if $p=\sum_{i,j=1}^n n^{-1}e_{ij}$, then $p$ is also a rank-one projection, and $\beta\big(\bigotimes_{k=1}^\infty a_k\big)=p\otimes\big(\bigotimes_{k=2}^\infty a_{k-1}\big)$.
\end{remark}

Although we found the endomorphism $\beta$ using our general construction, and were surprised to find it, we have now learned that other authors have shown that Cuntz-Krieger algebras can be realised as a Stacey crossed product by an endomorphism of the core, for example, \cite[Example~2.5]{Ror} and \cite{Kwa} (see Remark~\ref{remark-kwa} below). Now that we have found our $\beta$ we should be able to prove Proposition~\ref{CK=cp} directly. Before doing this we will revisit the need for the hypotheses on the graph $E$: that $E$ is row finite with no sources ensured that the  path space $E^\infty$ is locally compact, but a direct proof should  not go through $C_0(E^\infty)$.
Our formula for $\beta$ only makes sense when $E$ is column-finite and has no sinks: the coefficients in \eqref{defbeta3} are crucial, as we will see in the proof of the next result. It seems likely that column-finiteness is necessary. There is no obvious way to adjust for sinks, either: if we try to interpret empty sums as $0$, then $\beta(t_\mu t_\nu^*)$ would be $0$ if either $\mu$ or $\nu$ ends at a sink, but this property is not preserved by multiplication. (If $\nu$ ends at a sink but $\mu$ doesn't, then we'd have $\beta(t_\mu t_\nu^*)=0$ but $\beta((t_\mu t_\nu^*)(t_\nu t_\mu^*))=\beta(t_\mu t_\mu^*)\not=0$.) So the best we can do is the following.

\begin{thm}\label{genendodecomp}
Suppose that $E$ is a column-finite directed graph with no sinks, and $\{t_e,p_v\}$ is the universal Cuntz-Krieger $E$-family which generates $C^*(E)$. Then there is an endomorphism $\beta$ of the core $C^*(E)^\gamma$ such that
\begin{equation}\label{defbeta666}
\beta(t_\mu t_\nu^*)=\sum_{s(e)=r(\mu),\;s(f)=r(\nu)}\big(|s^{-1}(r(\mu))|\,|s^{-1}(r(\nu))|\big)^{-1/2}t_{e\mu} t_{f\nu}^*.
\end{equation}
The series 
\begin{equation*}
\sum_{e\in E^1}|s^{-1}(s(e))|^{-1/2}t_e
\end{equation*}
converges strictly in $M(C^*(E))$ to an isometry $W$  satisfying $\beta(t_\mu t_\nu^*)=Wt_\mu t_\nu^* W^*$. If $\iota$ is the inclusion of the core $C^*(E)^\gamma$ in $C^*(E)$, then the associated  representation $\iota\times W$ of the Stacey crossed product $C^*(E)^\gamma\times_{\beta}\N$ is an isomorphism  onto $C^*(E)$. 
 \end{thm}

\begin{proof}
For each $v\in E^0$, $\{t_\mu t_\nu^*:|\mu|=|\nu|=i, s(\mu)=s(\nu)=v\}$ is a set of matrix units for $\F_i(v)$ (see \cite[page~312]{BPRS}). We claim their images $e_{\mu,\nu}$ under $\beta$ in $\F_{i+1}(v)$, defined by the right-hand side of \eqref{defbeta666}, are also matrix units. The product $e_{\mu,\nu}e_{\kappa,\lambda}$ contains terms like $t_{e\mu}t_{f\nu}^*t_{g\kappa}t_{h\lambda}^*$, which is zero unless $f=g$ and $\nu=\kappa$. Since we then have $r(\nu)=r(\kappa)$, the two central terms in the coefficient are the same, and
\[
e_{\mu,\nu}e_{\kappa,\lambda}=\sum_{s(e)=r(\mu),\;s(f)=r(\nu),\;s(h)=r(\lambda)}|s^{-1}(r(\mu))|^{-1/2}|s^{-1}(r(\nu))|^{-1}|s^{-1}(r(\lambda))|^{-1/2}t_{e\mu}t_{h\lambda}^*.
\]
For fixed $e$ and $h$, there are $|s^{-1}(r(\nu))|$ edges $f$  with $s(f)=r(\nu)$, and for each of these the summand is exactly the same. So 
\[
e_{\mu,\nu}e_{\kappa,\lambda}=\sum_{s(e)=r(\mu),\;s(h)=r(\lambda)}|s^{-1}(r(\mu))|^{-1/2}|s^{-1}(r(\lambda))|^{-1/2}s_{e\mu}s_{h\lambda}^*=e_{\mu,\lambda}.
\]
(Notice that the coefficients in the definition of $e_{\mu,\nu}$ had to be just right for this to work.) Thus $\{e_{\mu,\nu}\}$ is a set of matrix units as claimed, and there is a well-defined homomorphism $\beta_i:\F_i(v)\to \F_{i+1}(v)$ satisfying \eqref{defbeta666} (well, with $\beta$ replaced by $\beta_i$). These combine to give a homomorphism $\beta_i$ of $\F_i=\bigoplus_v \F_i(v)$ into $\F_{i+1}=\bigoplus_v\F_{i+1}(v)$.

To define $\beta$ on $C_i:=\F_0+\F_1+\cdots +\F_i$, we take the $i$-expansion $c=\sum_{j=0}^ic_j$ described in Proposition~\ref{describecore}(b), and define
$\beta^i(c)=\sum_{j=0}^i\beta_j(c_j)$. The uniqueness of the $i$-expansion implies that this gives a well-defined function $\beta^i$ on each $C_i$; to check that they give a well-defined function on $\bigcup_{i=0}^\infty C_i$, we need to check that $\beta^i(c)=\beta^{i+1}(c)$ for $c\in C_i\subset C_{i+1}$. Suppose that $c\in C_i$ has $i$-expansion $c=\sum_{j=0}^ic_j$. Then the $(i+1)$-expansion is $c=\big(\sum_{j=0}^{i-1}c_j\big)+c_i'+d$, where $c_i'\in \EE_i$ and $d\in \F_i\cap \F_{i+1}$ are uniquely determined by $c_i=c_i'+d$. Lemma~\ref{idcap} implies that if $t_\mu t_\nu^*$ belongs to $\F_i\cap \F_{i+1}$, then $v:=s(\mu)=s(\nu)$ satisfies $0<|r^{-1}(v)|<\infty$, and two applications of the Cuntz-Krieger relation at $v$ show that
\begin{align*}
\beta_{i+1}(t_\mu t_\nu^*)&=\beta_{i+1}\Big(\sum_{r(g)=v}t_{\mu g} t_{\nu g}^*\Big)\\
&=\sum_{r(g)=v,\;s(e)=r(\mu g),\;s(f)=r(\nu g)}\big(|s^{-1}(r(\mu))|\,|s^{-1}(r(\nu))|\big)^{-1/2}t_{e\mu g} t_{f\nu g}^*\\
&=\sum_{s(e)=r(\mu),\;s(f)=r(\nu)}\big(|s^{-1}(r(\mu))|\,|s^{-1}(r(\nu))|\big)^{-1/2}t_{e\mu} t_{f\nu}^*\\
&=\beta_i(t_\mu t_\nu^*).
\end{align*}
Thus $\beta_i(c_i')+\beta_{i+1}(d)=\beta_i(c_i')+\beta_i(d)=\beta_i(c_i)$. Thus
\[
\beta^{i+1}(c)=\Big(\sum_{j=0}^{i-1} \beta_i(c_j)\Big)+\beta_i(c_i')+\beta_{i+1}(d)=\Big(\sum_{j=0}^{i-1} \beta_i(c_j)\Big)+\beta_i(c_i)=\beta^i(c).
\]
  
At this stage we have a well-defined map $\beta:\bigcup_{i=0}^\infty C_i\to C^*(E)^\gamma$ satisfying \eqref{defbeta666}. This map is certainly linear, and we need to prove that it is multiplicative. We consider $t_\mu t_\nu^*\in \F_i$ and $t_\kappa t_\lambda^*\in \F_j$, and we may as well suppose $i\leq j$. Then multiplying together the two formulas for $\beta^i(t_\mu t_\nu^*)$ and $\beta^j(t_\kappa t_\lambda^*)$ gives a linear combination of things like $t_{e\mu}t_{f\nu}^*t_{g\kappa} t_{h\lambda}^*$. Because $i\leq j$, this product is $0$ unless $g\kappa=f\nu\kappa'$, and then it is $t_{e\mu\kappa'}t_{h\lambda}^*$. So the sum collapses just as it did in the first paragraph, and we obtain the formula for 
\[
\beta(t_{\mu\kappa'}t_{\lambda}^*)=\beta((t_\mu t_\nu^*)(t_\kappa t_\lambda^*)).
\]
Thus $\beta$ is multiplicative. Since it is clearly $*$-preserving, it is a $*$-homomorphism, and as such is automatically norm-decreasing on each $C_i$. Thus $\beta$ extends to an endomorphism, also called $\beta$, of $\overline{\bigcup_{i=0}^\infty C_i}=C^*(E)^\gamma$. 

Next, note that for each $v\in E^0$, the partial isometries $\{t_e:s(e)=v\}$ have the same initial projection $p_v$, and mutually orthogonal range projections $t_et_e^*$, so\footnote{It is important here that there are only finitely many summands, so  column-finiteness is crucial.} $T_v:=\sum_{s(e)=v}|s^{-1}(v)|^{-1/2}t_e$ is a partial isometry with initial projection $p_v$ and range projection $\sum_{s(e)=v=s(f)}|s^{-1}(v)|^{-1}t_et_f^*$. Now the partial isometries $\{T_v:v\in E^0\}$ have mutually orthogonal initial projections and mutually orthogonal range projections, and hence their sum converges strictly to a partial isometry $W$ with initial projection $\sum_{v\in E^0}p_v=1_{M(C^*(E))}$. In other words, $W$ is an isometry.

The covariance relation $\beta(t_\mu t_\nu^*)=Wt_\mu t_\nu^* W^*$ is easy to check, and the universal property of the Stacey crossed product $(C^*(E)^\gamma\times_{\beta}\N,i_{C^*(E)^\gamma},v)$ gives a homomorphism $\iota\times W$ such that $(\iota\times W)\circ i_{C^*(E)^\gamma}=\iota$ and $(\iota\times W)(v)=W$. This homomorphism satisfies $(\iota\times W)\circ\hat\beta_z=\gamma_z\circ(\iota\times W)$, and since $\iota$ is faithful (being an inclusion), the dual-invariant uniqueness theorem (Proposition~\ref{diut}) implies that $\iota\times W$ is injective. The range contains each $t_e=|s^{-1}(s(e))|^{1/2}\iota(t_et_e^*)W$, and hence $\iota\times W$ is surjective.
\end{proof}

\begin{remark}\label{remark-kwa}
When the graph $E$ is finite and has no sinks, the endomorphism $\beta$ has also been found by Kwa\'sniewski \cite{Kwa}. He proves in \cite[Theorem~3.2]{Kwa} that $C^*(E)$ is isomorphic to a partial-isometric crossed product $C^*(E)^\gamma\rtimes_\beta\Z$ as introduced in \cite{ABL}. He then applies general results about partial-isometric crossed products from \cite{ABL} and \cite[\S1]{Kwa} to $(C^*(E)^\gamma,\beta)$, and recovers many of the main structure theorems for graph $C^*$-algebras, as they apply to finite graphs with no sinks \cite[\S3]{Kwa}.  For such graphs $E$, the endomorphism $\beta$ is conjugation by an isometry in $C^*(E)$, and Theorem~4.15 of \cite{ABL} implies that $C^*(E)^\gamma\rtimes_\beta\Z$ is isomorphic to the Exel crossed product $C^*(E)^\gamma\rtimes_{\beta,K}\N$. Since $\beta$ is injective, we can then deduce from \cite[Theorem~4.7]{E} that $C^*(E)$ is isomorphic to the Stacey crossed product, as in Theorem~\ref{genendodecomp}.
\end{remark}

\appendix

\section{The core in a graph algebra}\label{analysecore}

Suppose that $E$ is an arbitrary directed graph, which could have infinite receivers, infinite emitters, sources and/or sinks.  In Theorem~\ref{genendodecomp}, we wanted a description of the core $C^*(E)^\gamma$ which did not depend on row finiteness, and since we cannot recall seeing a suitable description in the literature, we give one here. As in the row-finite case, which is done in \cite{BPRS} and \cite{CBMS}, our description uses the subspaces
\[
\F_i:=\clsp\{ t_\mu t_\nu^*:|\mu|=|\nu|=i,\ s(\mu)=s(\nu)\},
\]
which one can easily check are in fact $C^*$-subalgebras. (By convention, $\F_0=\clsp\{p_v:v\in E^0\}$.)

We refer to vertices which are either infinite receivers or sources as ``singular vertices''. Recall that no Cuntz-Krieger relation is imposed at a singular vertex $v$, but if $v$ is an infinite emitter then we impose the inequality $p_v\geq \sum_{e\in F}t_et_e^*$ for every finite subset $F$ of $r^{-1}(v)$.

\begin{prop}\label{describecore}
Let $E$ be a directed graph, and define $\F_i$ as above.

\smallskip
\textnormal{(a)} For $i\geq 0$, $C_i:=\F_0+\F_1+\cdots+\F_i$ is a $C^*$-subalgebra of $C^*(E)^\gamma$, $C_i\subset C_{i+1}$ and $C^*(E)^\gamma=\overline{\bigcup_{i=0}^\infty C_i}$.

\smallskip
\textnormal{(b)} For each $i\geq 0$ and each $c\in C_i$, there are unique elements
\[
c_j\in \EE_j:=\clsp\{t_\mu t_\nu^*:|\mu|=|\nu|=j \text{ and }\ s(\mu)=s(\nu)\text{ is singular}\}
\]
for $0\leq j<i$ and $c_i\in\F_i$ such that $c=\sum_{j=0}^i c_j$.
\end{prop}

Since $C_i\subset C_j$ for $i<j$, an element of $C_i$ has lots of the expansions described in (b). We refer to the one obtained by viewing $c$ as an element of $C_j$ as the \emph{$j$-expansion} of~$c$. 

\begin{proof}[Proof of Proposition~\ref{describecore}\textnormal{(a)}]
Since $\bigcup_i C_i$ is a vector space containing every element of the form $t_\mu t_\nu^*$ with $|\mu|=|\nu|$, it is dense in $C^*(E)^\gamma$, and we trivially have $C_i\subset C_{i+1}$. We prove that $C_i$ is a $C^*$-subalgebra by induction on $i$. For $i=0$, it's straightforward to see that $C_0=\F_0$ is a $C^*$-subalgebra. Suppose that $C_i$ is a $C^*$-subalgebra. If $|\kappa|=|\lambda|=i+1$ and $|\mu|=|\nu|\leq i+1$, then the formula
\[
(t_\mu t_\nu^*)(t_\kappa t_\lambda^*)=
\begin{cases}
0&\text{ unless $\kappa$ has the form $\nu\kappa'$}\\
t_{\mu\kappa'}t_\lambda^*&\text{if $\kappa=\nu\kappa'$}
\end{cases}
\]
shows that $\F_{i+1}$ is an ideal in the $C^*$-subalgebra $C^*(C_{i+1})$ of $C^*(E)^\gamma$ generated by $C_{i+1}$. Since $C_i$ is a $C^*$-subalgebra of $C^*(C_{i+1})$, the sum $C_{i+1}=C_i+\F_{i+1}$ is also a $C^*$-subalgebra (and in fact $C^*(C_{i+1})=C_i+\F_{i+1}=C_{i+1}$).
\end{proof}

For part (b), we need to do some preparation. The standard argument of \cite[\S2]{BPRS} or \cite[Chapter~3]{CBMS} shows that, for fixed $i$ and $v\in E^0$, 
\[
\{t_\mu t_\nu^*:|\mu|=|\nu|=i,\ s(\mu)=s(\nu)=v\}
\]
is a set of non-zero matrix units, and hence their closed span $\F_i(v)$ is a $C^*$-subalgebra of $\F_i$ isomorphic to $\K(\ell^2(E^i\cap s^{-1}(v)))$. Since $\F_i(v)\F_i(w)=\{0\}$ for $v\not=w$, $\F_i$ is the $C^*$-algebraic direct sum $\bigoplus_{v\in E^0}\F_i(v)$. For $j$ satisfying $0\leq j\leq i$, we set
\[
D_{j,i}:=\F_j+\cdots+\F_i=\clsp\{ t_\mu t_\nu^*:j\leq |\mu|=|\nu|\leq i,\ s(\mu)=s(\nu)\},
\]
which is another $C^*$-subalgebra by the argument in the previous proof. It is an ideal in $C_i$. Now we prove a lemma.

\begin{lemma}\label{idcap}
For every $i\geq 1$ and every $j<i$, we have
\[
\F_j(v)\cap D_{j+1,i}=
\begin{cases}
0&\text{if $v$ is a singular vertex}\\
\F_j(v)&\text{if $0<|r^{-1}(v)|<\infty$.}
\end{cases}
\]
\end{lemma}

\begin{proof}
The result is trivially true if $\F_j(v)$ is $\{0\}$, so we suppose it isn't. Since $\F_j(v)\cap D_{j+1,i}$ is an ideal in $\F_j(v)$, it is either $\{0\}$ or $\F_j(v)$.    
First suppose that $v$ is not a singular vertex. Then the Cuntz-Krieger relation at $v$ implies that $\F_j(v)\subset D_{j+1,i}$, and hence $\F_j(v)\cap D_{j+1,i}=\F_j(v)$.

Now suppose that $\F_j(v)\cap D_{j+1,i}=\F_j(v)$; we will show that $v$ is not singular. Choose $\mu\in $$E^j$ with $s(\mu)=v$. Then $t_\mu t_\mu^*$ is a non-zero element of element of $D_{j+1,i}$, so there exist $\kappa$ and $\lambda$ satisfying $j+1\leq|\kappa|=|\lambda|\leq i$ and $(t_\mu t_\mu^*)(t_\kappa t_\lambda^*)\not= 0$. But this implies that $\kappa$ has the form $\mu\kappa'$, and $v=s(\mu)$ cannot be a source. It also implies that $t_\mu t_\mu^*t_\kappa t_\kappa^*$ is non-zero, and hence so is the larger projection $t_\mu t_\mu^* t_{\mu\kappa_{j+1}}t_{\mu\kappa_{j+1}}^*$. Thus $\F_j(v)\cap \F_{j+1}\neq \{0\}$, and since $\F_j(v)\cap \F_{j+1}$ is  an  ideal in $\F_j(v)$, it follows that $\F_j(v)\cap \F_{j+1}=\F_j(v)$.

We now know that the projection $t_\mu t_\mu^*$ belongs to $\F_{j+1}$, and hence is a projection in a $C^*$-algebraic direct sum $\bigoplus_{w\in E^0}\F_{j+1}(w)$. The norms of elements in this direct sum are arbitrarily small off finite subsets of $E^0$, and projections have norm $0$ or $1$, so there are a finite subset $F$ of $E^0$ and projections $q_w\in \F_{j+1}(w)$ such that $t_\mu t_\mu^*=\sum_{w\in F}q_w$. Each $q_w$ is a projection in $\F_{j+1}(w)=\K(\ell^2(E^{j+1}\cap s^{-1}(w)))$, and hence has finite trace. Thus $t_\mu t_\mu^*$ has finite trace. On the other hand, for every edge $e$ with $r(e)=s(\mu)=v$, we have $t_\mu t_\mu^*\geq t_{\mu e}t_{\mu e}^*$; since $\{t_{\mu e}t_{\mu e}^*:r(e)=s(\mu)\}$ is a family of mutually orthogonal projections of trace $1$, we have
\[
\Tr (t_\mu t_\mu^*)\geq \sum_{r(e)=s(\mu)}\Tr(t_{\mu e}t_{\mu e}^*)=|r^{-1}(s(\mu))|=|r^{-1}(v)|.
\]
Since we have already eliminated the possibility that $v$ is a source, this proves that it is not a singular vertex, as required.
\end{proof}

\begin{proof}[Proof of Proposition~\ref{describecore}\textnormal{(b)}]
Lemma~\ref{idcap} implies that
\[
\F_j\cap D_{j+1,i}=\bigoplus\{\F_j(v):0<|r^{-1}(v)|<\infty\},
\]
and that $\F_j$ is the direct sum of $\F_j\cap D_{j+1,i}$ and 
\[
\EE_j=\bigoplus\{\F_j(v):\text{$v$ is a singular vertex}\}.
\]
Since $D_{j,i}=\F_j+D_{j+1,i}$, we have
\begin{align*}
C_i&=\F_0+D_{1,i}=(\EE_0+(\F_0\cap D_{1,i}))+D_{1,i}=\EE_0+D_{1,i}\\
&=\EE_0+(\EE_1+(\F_1\cap D_{2,i}))
=\EE_0+\EE_1+D_{2,i}\\
&\ \vdots \\
&=\EE_0+\cdots \EE_{i-1}+D_{i,i}=\EE_0+\cdots \EE_{i-1}+\F_i,
\end{align*}
which shows that $c$ has the claimed expansion.

To establish uniqueness, suppose that $c_j,d_j\in \EE_j$ for $j<i$, that $c_i, d_i\in \F_i$, and that $\sum_{j=0}^ic_j=\sum_{j=0}^id_j$. Then $c_0-d_0=\sum_{j=1}^i(d_j-c_j)$ belongs to $\EE_0\cap \F_0$, because the left-hand side does, and to $D_{1,i}$, because the right-hand side does. Since $\F_0=\EE_0\oplus(\F_0\cap D_{1,i})$, we have $\EE_0\cap (\F_0\cap D_{1,i})=\{0\}$, and we deduce that $c_0=d_0$ and $\sum_{j=1}^ic_j=\sum_{j=1}^id_j$. Now an induction argument using $\EE_j\cap (F_j\cap D_{j+1,i})=\{0\}$ gives the result. 
\end{proof}

\section{A question posed by Ionescu and Muhly}\label{sec-IMqu}
Let $\sigma$ be a local homeomorphism of a compact Hausdorff space $X$ and let $G$ be the Deaconu-Renault groupoid described in \S\ref{sec-IM}.  When $M_L$ has an orthonormal basis $\{m_i\}_{i=1}^n$,  Ionescu and Muhly ask on page~201 of \cite{IM} if  $C^*(G)$ is isomorphic to a  Stacey multiplicity-$n$ crossed product involving the $n$ isometries $\{k_{M_L}(m_i)\}_{i=1}^n$. Their question was prompted by their Theorem~4.3. In this section we show that the answer to this question is  negative (see Example~\ref{IM-qu}).  We start by generalising \cite[Theorem~4.3]{IM}.

Recall that a finite set $\{m_i:1\leq i\leq n\}$ in a right Hilbert $A$-module $M$ is a \emph{Parseval frame} if
\begin{equation}\label{defParseval}
m=\sum_{i=1}^nm_i\cdot\langle m_i, m\rangle_A\ \text{ for every $m\in M$;}
\end{equation}
equivalently, $\{m_i\}$ is a Parseval frame if and only if the finite-rank operator $\sum_i\Theta_{m_i,m_i}$ is the
identity operator $1\in \L(M)$. 
The discussion in \cite[page~5]{PR1} shows that $M$ has a (finite) Parseval frame exactly when $M$ is finitely generated and projective. 

\begin{prop}\label{prop-universal}
Let $(A,\alpha,L)$ be an Exel system with $A$ unital. Suppose $\{m_i\}_{i=1}^n$ is a Parseval frame for the associated right-Hilbert bimodule $M_L$, and set $s_i=k_{M_L}(m_i)$ for $1\leq i\leq n$.
\begin{enumerate}
\item\label{prop-universal-a} The set $k_A(A)\cup\{s_i:1\leq i\leq n\}$ generates $A\rtimes_{\alpha,L}\N=\O(M_L)$, $\sum_{i=1}^n s_is_i^*=1$ and  
\begin{equation}\label{H1forjS}
s_i^*k_A(a)s_j=k_A(\langle m_i\,,\, a\cdot m_j\rangle_L)\quad\text{for $a\in A$.}
\end{equation}
\item\label{prop-universal-b} Suppose that $\pi$ is a unital representation of $A$ on $\H$, and that $\{S_i:1\leq i\leq n\}\subset B(\H)$ satisfies $\sum_{i=1}^n S_iS_i^*=1$ and
\begin{equation*}\label{eq-family}
S_i^*\pi(a)S_j=\pi(\langle m_i\,,\, a\cdot m_j\rangle_L)\quad\text{for $a\in A$.}
\end{equation*}
Then there is a representation $\pi\rtimes\{S_i\}$ of $A\rtimes_{\alpha,L}\N$ on  a Hilbert space $\H$ such that $\pi\rtimes\{S_i\}\circ k_A=\pi$ and $\pi\rtimes\{S_i\}(s_i)=S_i$ for $1\leq i\leq n$.

\item\label{prop-universal-c}  If $A$ is commutative, then $\alpha(a)\cdot m=m\cdot a$ for all $a\in A$ and $m\in M_L$, and
\begin{equation}\label{Staceycovariance}
k_A(\alpha(a))=\sum_{j=1}^n s_jj_A(a)s_j^*.
\end{equation}
\end{enumerate}
\end{prop}

\begin{proof}
A straightforward calculation using the reconstruction formula \eqref{defParseval} shows that for every $a\in A$ we have
\begin{equation}\label{LHactionFR}
\phi(a)=\sum_{i=1}^n\Theta_{a\cdot m_i,m_i}
\end{equation}
as operators on $M_L$; thus $\phi^{-1}(\K(M_L))=A=K_\alpha$, and 
$\O(M_L)=\O(K_\alpha, M_L)=A\rtimes_{\alpha,L}\N$.

We have
\begin{align*}
\sum_{i=1}^n s_is_i^*&=\sum_{i=1}^nk_{M_L}(m_i)k_{M_L}(m_i)^*=\sum_{i=1}^n(k_{M_L},\pi)^{(1)}(\Theta_{m_i,m_i})\\
&=(k_{M_L},k_A)^{(1)}(\phi(1))=k_A(1)=1.
\end{align*}
For  $a\in A$ we have
\[
s_i^*k_A(a)s_j=k_{M_L}(m_i)^*k_A(a)k_{M_L}(m_j)=k_{M_L}(m_i)^*k_{M_L}(a\cdot m_j)=k_A(\langle m_i\,,\, a\cdot m_j\rangle_L).
\]
This gives (\ref{prop-universal-a}).

Now suppose we have $(\pi, \{S_i\}_{i=1}^n)$ as in (\ref{prop-universal-b}), and define $\psi:M_L\to B(\H)$ by
\[
\psi(m)=\sum_{i=1}^n S_i\pi(\langle m_i\,,\, m\rangle_L).
\]
We claim that that $(\psi, \pi)$ is a  representation of $M_L$ on $\H$. For $m,n\in M_L$ and $a\in A$ we have
\[
\psi(m\cdot a)=\sum_{i=1}^n S_i\pi(\langle m_i\,,\, m\cdot a\rangle_L)=\sum_{i=1}^n S_i\pi(\langle m_i\,,\, m\rangle_La)=\psi(m)\pi(a),
\]
\begin{align*}
\psi(m)^*\psi(n)&=\sum_{i,j=1}^n \pi(\langle m_i\,,\, m\rangle_L)^*S_i^*S_j\pi(\langle m_j\,,\, n\rangle_L)\\
&=\sum_{i,j=1}^n \pi(\langle m_i\,,\, m\rangle_L)^*\pi(\langle m_i\,,\,m_j\rangle_L)\pi(\langle m_j\,,\, n\rangle_L)\\
&=\sum_{i,j=1}^n\pi\big(\big\langle m_i\cdot\langle m_i\,,\, m\rangle_L\,,\,m_j\cdot \langle m_j\,,\, n\rangle_L\big\rangle_L\big)\\
&=\pi(\langle m\,,\, n\rangle_L),
\end{align*}
and 
\begin{align*}
\psi(a\cdot m)&=\sum_{i=1}^n S_i\pi(\langle m_i\,,\, a\cdot m\rangle_L)
=\sum_{i=1}^n S_i\pi\Big( \Big\langle m_i\,,\, a\cdot\sum_{j=1}^n m_j\cdot\langle m_j\,,\, m\rangle_L\Big\rangle_L \Big)\\
&=\sum_{i,j=1}^n S_i\pi(\langle m_i\,,\, a\cdot m_j\rangle_L\langle m_j\,,\, m\rangle_L)
=\sum_{i,j=1}^n S_iS_i^*\pi(a)S_j\pi(\langle m_j\,,\, m\rangle_L)\\
&=\Big(\sum_{i=1}^nS_iS_i^*\Big)\Big(\sum_{j=1}^n\pi(a)S_j\pi(\langle m_j\,,\, m\rangle_L)\Big)=\pi(a)\psi(m).
\end{align*}
Thus $(\psi,\pi)$ is a  representation of $M_L$ as claimed. We use the description of $\phi(a)$ in \eqref{LHactionFR} to compute $(\psi,\pi)^{(1)}(\phi(a))$:
\begin{align*}
(\psi,\pi)^{(1)}(\phi(a))
&=(\psi,\pi)^{(1)}\Big(\sum_{i=1}^n  \Theta_{a\cdot m_i, m_i}\Big)
=\sum_{i=1}^n \psi(a\cdot m_i)\psi(m_i)^*\\
&=\sum_{i=1}^n\pi(a)\psi(m_i)\psi(m_i)^*=\pi(a)\Big(\sum_{i=1}^n S_iS_i^*\Big)=\pi(a),
\end{align*}
and so $(\psi,\pi)$ is Cuntz-Pimsner covariant.

The universal property of $\O(M_L)$ gives a representation $\pi\rtimes\psi$ of $A\rtimes_{\alpha,L}\N=\O(M_L)$ such that $(\pi\rtimes\psi)\circ k_A = \pi$ and
\begin{align*}
(\pi\rtimes\psi)(s_i)
&=\psi(m_i)=\sum_{i=1}^nS_j\pi(\langle m_j\,,\,m_i\rangle_L)
=\sum_{i=1}^nS_jS_j^*S_i=S_i.
\end{align*}
Thus $\pi\rtimes\{S_i\}:=\psi\times\pi$ has the required properties, and we have proved (\ref{prop-universal-b}).

For (\ref{prop-universal-c}),  suppose that $A$ is commutative. For $a\in A$ and $q(b)\in q(A)\subset M_L$, we have
\begin{equation}\label{relforcommA}
\alpha(a)\cdot q(b)=q(\alpha(a)b)=q(b\alpha(a))=q(b)\cdot a,
\end{equation}
and this extends by continuity to $\alpha(a)\cdot m=m\cdot a$ for $m\in M_L$. Now for $a\in A$ we have
\begin{align*}
k_A(\alpha(a))&=\sum_{i,j=1}^n s_is_i^*k_A(\alpha(a))s_js_j^* =\sum_{i,j=1}^n s_ik_A(\langle m_i\,,\, \alpha(a)\cdot m_j\rangle_Ls_j^*\\
&=\sum_{i,j=1}^n s_ik_A(\langle m_i\,,\,  m_j\cdot a\rangle_Ls_j^*=\sum_{i,j=1}^n S_ik_A(\langle m_i\,,\,  m_j\rangle_L)k_A(a)s_j^*\\
&=\sum_{i,j=1}^n s_is_i^*s_jk_A(a)s_j^*=\sum_{j=1}^n s_jk_A(a)s_j^*,
\end{align*}
and we have proved part (\ref{prop-universal-c}).
\end{proof}

\begin{remark} To recover Theorem~4.3 of \cite{IM} from Proposition~\ref{prop-universal}, consider the Exel system $(C(X),\alpha,L)$  of Theorem~\ref{IMthm1}, and suppose that $\{m_i\}_{i=1}^n$ is an orthonormal basis for $M_L$. Then taking $a=1$ in \eqref{H1forjS} shows that each $s_i=k_{M_L}(m_i)$ is an isometry, and Proposition~\ref{prop-universal}(\ref{prop-universal-a}) says that $\{s_i\}_{i=1}^n$ is a Cuntz family.  
Equation~\eqref{Staceycovariance} is property (3) of $C^*(G)$ in \cite[Theorem~4.3]{IM}. This immediately implies property (2) in \cite[Theorem~4.3]{IM} (just multiply \eqref{Staceycovariance} on the left by $s_i$); property (1), which concerns the operator $S=k_{M_L}(1)$, follows from \eqref{relforcommA}:
\[
k_A(\alpha(a))S=k_A(\alpha(a))k_{M_L}(q(1))=k_{M_L}(\alpha(a)\cdot q(1))=k_{M_L}(q(1)\cdot a)=Sk_A(a).
\]
\end{remark}

\begin{lemma}\label{lem-basis} Let $X=\T$ and $\sigma:\T\to \T: z\mapsto z^2$, and consider the Exel system $(C(\T),\alpha, L)$ of Theorem~\ref{IMthm1}.  Let $U$ be  a small neighbourhood of $1$ in $\T$. There is an orthonormal basis $\{m_0, m_1\}$  of $M_L$ such that $m_0$ is identically $\sqrt{2}$ on $U$ and $m_1$ is identically $0$ on $U$.  
\end{lemma}

\begin{proof}
A function $m_0$ satisfying 
\begin{equation*}\label{filter}
|m_0(w)|^2+|m_0(-w)|^2=2 \text{\ for $w\in\T$}.
\end{equation*}
 is called a \emph{quadrature mirror filter}. The \emph{conjugate mirror filter} $m_1$, defined by  $m_1(z)=z\overline{m_0(-z)}$, also satisfies \eqref{filter}. It is then straightforward to check that $\{m_0, m_1\}$ is orthonormal and that $m=m_0\cdot\langle m_0\,,\, m\rangle_L+ m_1\cdot\langle m_1\,,\, m\rangle_L$ for all $m\in M_L$ (see, for example, \cite[Theorem~1]{PR1}). Thus  $\{m_0, m_1\}$ is an orthonormal basis. If  $m_0(u)=\sqrt{2}$ for $u\in U$ then \eqref{filter} implies that $m_1(u)=u\overline{m_0(-u)}=0$. So it suffices to construct a suitable $m_0$.
For this, choose a continuous function $g:[-\frac{\pi}{2},\frac{\pi}{2}]\to \R$ such that $g(\theta)=1$ for $\theta$ near $0$ and $g(-\frac{\pi}{2})=g(\frac{\pi}{2})=\frac{1}{\sqrt{2}}$, and for $-\pi\leq \theta\leq \pi$ define
\[
m_0(e^{i\theta})=\begin{cases}
\sqrt{2}\sqrt{1-g(\theta+\pi)^2}&\text{ for $\theta\in[-\pi,-\frac{\pi}{2})$}\\
\sqrt{2}\sqrt{1-g(\theta-\pi)^2}&\text{ for $\theta\in(\frac{\pi}{2},\pi]$}\\
\sqrt{2}g(\theta)				&\text{ for $\theta\in [-\frac{\pi}{2},\frac{\pi}{2}]$}.
\end{cases}\qedhere
\]
\end{proof}

We will use the next lemma in  Example~\ref{IM-qu} to obtain a contradiction when we assume there that a particular $C^*(G)$ is a Stacey crossed product of multiplicity $n$.

\begin{lemma}\label{lem-notStacey}
Let $(A,\alpha,L)$ be an Exel system in which $A$ is unital and commutative, and $M_L$ has an orthonormal basis $\{m_i\}_{i=1}^n$.  Assume that $(\O_{M_L},\{S_i\}):=(\O_{M_L},\{k_M(m_i)\})$ is a Stacey crossed product of multiplicity $n$ for $(A,\alpha)$. Let $\{n_i\}_{i=1}^n$ be another orthonormal basis for $M_L$.   Then 
\begin{equation}\label{equalcoeff}
\langle n_i\,,\, a\cdot n_j\rangle_L=\langle m_i\,,\, a\cdot m_j\rangle_L\quad\text{for $a\in A$.}
\end{equation}
\end{lemma}

\begin{proof}
Since $\{n_i\}$ is orthonormal, we have 
\[
k_{M_L}(n_i)^*k_{M_L}(n_i)^*=k_A(\langle n_i,n_i\rangle)=k_A(1_A)=1_{\O(M_L)}.
\]
Thus  Proposition~\ref{prop-universal} says that $\{T_i:=k_M(n_i)\}$ is a Cuntz family in $\O_{M_L}$ such that
\begin{equation}\label{H1forT}
T_i^*k_A(a)T_j=k_A(\langle n_i\,,\, a\cdot n_j\rangle_L)\quad\text{and}\quad
k_A(\alpha(a))=\sum_{i=1}^n T_ik_A(a)T_i^* \quad\text{for $a\in A$.}
\end{equation}
Since  $(\O_{M_L},\{S_i\})$ is by assumption a Stacey crossed product of multiplicity $n$, there is a homomorphism $\rho:\O_{M_L}\to \O_{M_L}$ such that $\rho\circ k_A=k_A$ and $\rho(S_i)=T_i$. Applying $\rho$ to
\eqref{H1forjS} implies that
\[
T_i^*k_A(a)T_j=k_A(\langle m_i\,,\, a\cdot m_j\rangle_L)\quad\text{for $a\in A$,}
\]
which in view of \eqref{H1forT} gives \eqref{equalcoeff}.
\end{proof}

Equation \eqref{equalcoeff} holds, for example, if $a=\alpha(b)$: then
\begin{align*}
\langle n_i\,,\, a\cdot n_j\rangle_L=\langle n_i\,,\, n_j\cdot b\rangle_L
&=\langle n_i\,,\, n_j\rangle_Lb
=\langle m_i\,,\, m_j\rangle_Lb=\langle m_i\,,\, a\cdot m_j\rangle_L.
\end{align*}
So to find somewhere it does not hold, we need to look at elements $a$ which are not in the range of $\alpha$.

\begin{example}\label{IM-qu}
Let  $A=C(\T)$, $\alpha(f)(z)=f(z^2)$ and $L(f)(z)=\frac{1}{2}(f(w)+f(-w))$ where $w^2=z$. Let $U$ be a neighbourhood of $1$ in $\T$ such that $z\mapsto z^2$ is injective on $U$, and let $\{m_0, m_1\}$ be the orthonormal basis from Lemma~\ref{lem-basis}. Let $f$ be a nonzero  function in $C(\T)$ with support in $U$. Then  $f\cdot m_0=\sqrt{2}f$ and $f\cdot m_1=0$. We immediately have $\langle m_1,f\cdot m_1\rangle_L=0$. Both summands in
\[
\langle m_0,f\cdot m_0\rangle_L(z)=\frac{1}{2}\big(\overline{m_0(w)}\sqrt{2}f(w)+\overline{m_0(-w)}\sqrt{2}f(-w)\big)\ \text{ where $w^2=z$}
\]
vanish unless $z=w^2=(-w)^2$ is in $U^2$; for such $z$, we can choose the square root $w$ in $U$, and then $f(-w)=0$, so
\[
\langle m_0,f\cdot m_0\rangle_L(z)=
\begin{cases}
f(w)&\text{ where $w\in U$ and $w^2=z$}\\
0&\text{ if $z\notin U^2$.}
\end{cases}
\]
Now to see that \eqref{equalcoeff} does not hold, note that $\{n_0,n_1\}:=\{m_1,m_0\}$ is also an orthonormal basis for $M_L$, and $\langle n_0,f\cdot n_0\rangle_L=\langle m_1,f\cdot m_1\rangle_L=0$ is not the same as $\langle m_0,f\cdot m_0\rangle_L$.
By Lemma~\ref{lem-notStacey},  $(\O(M_L),\{S_i\}):=(\O(M_L),\{k_M(m_i)\})$ is not the Stacey multiplicity-$n$ crossed product. Thus $C^*(G)\cong \O(M_L)$ is not isomorphic to this crossed product either. 
\end{example}


\end{document}